\documentclass{article}
\usepackage{amssymb}
\usepackage{amsmath}
\usepackage{amsthm}
\usepackage{epsfig}
\usepackage{mathrsfs}
\usepackage{graphicx}
\usepackage{xcolor}
\usepackage{bbm}
\usepackage{hyperref}
\usepackage{cleveref}
\usepackage{color}
\usepackage{float}

\newtheorem{theorem}{Theorem}[section]
\newtheorem{lemma}[theorem]{Lemma}
\newtheorem{proposition}[theorem]{Proposition}
\newtheorem{definition}[theorem]{Definition}
\newtheorem{remark}[theorem]{Remark}
\newtheorem{corollary}[theorem]{Corollary}

\def\N{{\mathbb N}}
\def\Z{{\mathbb Z}}
\def\R{{\mathbb R}}
\def\T{{\mathbb T}}
\def\W{{\mathcal{W}^2(\mathbb{T})}}

\def\to{\rightarrow}

\def\eps{\varepsilon}

\newcommand{\argmin}{\operatornamewithlimits{argmin}}
\newcommand{\argmax}{\operatornamewithlimits{argmax}}

\title{Convergence of a one dimensional Cahn-Hilliard equation with degenerate mobility}
\author{Matias G. Delgadino}
\date{\today}

\begin{document}

\maketitle
\begin{abstract}
We consider a one dimensional periodic forward-backward parabolic equation, regularized by a non-linear fourth order term of order $\eps^2\ll 1$. This equation is known in the literature as Cahn-Hilliard equation with degenerate mobility. Under the hypothesis of the initial data being well prepared, we prove that as $\eps\to0$, the solution converges to the solution of a well-posed degenerate parabolic equation. The proof exploits the gradient flow nature of the equation in $\W$ and utilizes the framework of convergence of gradient flows developed by Sandier-Serfaty (\cite{sandier2004gamma}, \cite{serfaty2011gamma}). As an incidental, we study fine energetic properties of solutions to the Thin-film equation $\partial_t\nu=(\nu\nu_{xxx})_x$.
\end{abstract}
\section{Introduction}
Given a smooth non-convex potential $W:\R_+\to\R$, we are interested in the properties of solutions $\nu^\eps$ to
\begin{equation}\label{eq:gradeps}
\left\{\begin{array}{clr}
\partial_t\nu&=(\nu (W'(\nu)-\epsilon^2\nu_{xx})_x)_x&\mbox{in $(0,\infty)\times\T$}\\
\nu(0)&=\nu_i^\eps&\mbox{on $\{0\}\times\T$}.
\end{array}\right.
\end{equation}
and more specifically, in their behavior as $\eps\to0^+$, where $\T$ denotes the one-dimension flat torus $\R/\Z$.	

Equation~\eqref{eq:gradeps} is known in the literature as the Cahn-Hilliard equation (\cite{cahn1958free},\cite{novick1984nonlinear}, \cite{lisini2012cahn}). The function $\nu^\eps$ models the concentration of one of two phases in a system undergoing phase separation. Mathematically, this equation could be considered as a fourth order regularization of a forward-backward parabolic equation, by the fourth order term $-\eps^2(\nu\nu_{xxx})_x$. In the case where $W$ vanishes identically, we are left with a fourth order parabolic equation
$$
\partial_t \nu=(m(\nu)\nu_{xxx})_x
$$
known as the Thin-film equation, with mobility $m(\nu)=\nu$, which is interesting on its own (see for instance
\cite{kinderlehrer1980introduction}, \cite{mellet2015thin}, \cite{bertozzi1996lubrication}). Note that, the Dirichlet Energy is a Lyapunov functional and that when $m(\nu)=\nu$, the equation is formally the gradient flow of this Energy under the $\W$ metric. This observation was made in the seminal paper by Otto \cite{otto1998lubrication} and has been exploited for some generalizations in \cite{dolbeault2009new}, \cite{matthes2009family} and \cite{lisini2012cahn}. 

The main result of this paper is the fact that, under some assumptions (see Theorem~\ref{thm:conv}), $\nu^\eps$ converges, as $\eps\to0$, to the unique solution $\nu_0$ of the following well-posed degenerate parabolic equation
\begin{equation}\label{eq:grad}
\left\{\begin{array}{cl}
\partial_t\nu&=(\nu (W^{**\prime}(\nu))_x)_x\\
\nu(0)&=\nu_i,
\end{array}\right.
\end{equation}
where $W^{**}$ denotes the convex envelope of $W$.

The mathematical intuition behind this convergence comes from the fact that, formally at least, we know that $\nu^\eps$ is the gradient flow of
\begin{equation}\label{eq:Eeps1}
E^\eps[\mu]=\int_\T\frac{\eps^2}{2}|\mu_x|^2+W(\mu)\;dx,
\end{equation}
while $\nu_0$ is the gradient flow of
$$
E^{**}[\mu]=\int_\T W^{**}(\mu)\;dx,
$$
with respect to $\W$, and it is somewhat classical that the energy $E^\eps$ $\Gamma$-converges to $E^{**}$ in $\W$. Unfortunately, it is well known that the  $\Gamma$-convergence of the energy is not enough to prove the convergence of the gradient flows.

Indeed, to be able to prove the convergence of the gradient flows we need an additional condition on the gradient of the energy. A sufficient condition for Hilbert spaces was given in the paper by Sandier and Serfaty \cite{sandier2004gamma}, which was later extended to metric spaces by Serfaty in \cite{serfaty2011gamma}. This additional condition is usually written as follows:
\begin{equation}\label{eq:gammaliminf}
\Gamma-\liminf_{\eps\to0^+}|\nabla E^\eps|\ge|\nabla E^{**}|\mbox{ (see Section~\ref{sec:notations} for definitions)},
\end{equation}
and proving this inequality is always the hard part of the Sandier-Serfaty approach. However, in our case $|\nabla E^\eps|$ is not well understood, so we need to introduce a different quantity for which we prove a condition similar to \eqref{eq:gammaliminf} (see Theorem~\ref{thm:lscgrad}).

The framework of Sandier-Serfaty has been applied to an array of diverse problems. To name a few we have: Allen-Cahn \cite{mugnai2007allen}, Cahn-Hilliard \cite{le2008gamma}, \cite{bellettini2012convergence}, non-local interactions energies \cite{craig2015convergence}, TV flow \cite{colombo2012passing}  and Fokker Plank \cite{arnrich2012passing}. The most relevant reference for this paper was written by Belletini, Bertini, Mariani and Novaga \cite{bellettini2012convergence}, where they consider the convergence of the one dimensional Cahn-Hilliard equation on the Torus with mobility equal to one:
\begin{equation}\label{eq:belletini}
\partial_t\nu=(W'(\nu)-\epsilon^2\nu_{xx})_{xx}.
\end{equation}
We actually borrow some of the notations and the ideas on how to track the oscillations of the solution. The main difference between \cite{bellettini2012convergence} and our work is that \eqref{eq:belletini} is a gradient flow of \eqref{eq:Eeps1} in the Hilbert space $H^{-1}(\T)$, instead of the metric space $\W$. Besides bringing some non-trivial technical issues, working with a degenerate mobility coefficient also means that the estimates degenerate when the solution is near zero; this actually turns out to be a major issue that keeps showing up in the Thin Film equation literature as well.

In the spirit of being self-contained, we give a brief introduction into the framework developed by Ambrosio, Gigli and Savare for gradient flows in $\W$ (see Appendix~\ref{ap:wasserstein}). We try to outline all of the tools and terminology used in the paper, but it is in no way complete and the interested reader should definitely take the time to read \cite{ambrosio2006gradient}. 

A word of caution is that the framework developed in \cite{ambrosio2006gradient} can not be applied to the functional $E^\eps$, as it is neither $\lambda$-convex in the sense given at \cite{ambrosio2006gradient} (or the relaxed notion \cite{kamalinejad2012optimal}), nor regular. In fact, the subdifferential of $E^\eps$ is not really well understood; no matter how regular the measure is, if it vanishes at some point, it has not been proven that the natural candidate is indeed a subdifferential. 

We deal with this setback by considering Otto's approach in \cite{otto1998lubrication}, which constructs solutions to the equation as the limit of Minimizing Movements, an idea that was originated by De Giorgi. In the case of $E^\eps$, this has been made rigorous in \cite{lisini2012cahn}, where the authors are even able to prove a uniform $L^2_t(H^2_x)$ estimate for the constant interpolant of the discrete approximations, by using a discrete version of the entropy dissipation inequality (for the continuum case see \cite{bertozzi1996lubrication}). In this paper we go a bit further and we obtain an Energy inequality (see \eqref{eq:epsgradflow}), by defining a non-standard functional $\mathcal{G}^\eps$ (see Section~\ref{sec:notations}), which we prove to be lower semicontinuous in $H^2$ (see Lemma~\ref{lem:lscG}) and which agrees with the size of the subdifferential, when we know $E^\eps$ to be strongly subdifferentiable and $\mu$ to be regular enough. To our knowledge, this is a completely novel result in the literature and gives a starting point to understand the $\W$ gradient flows of energies involving derivatives. Shedding some light onto this topic will be part of the author's upcoming work.

Once we are able to prove the existence of an appropriate solution to our equation, the main obstacle we encounter, when we try to prove the convergence, is oscillatory behavior, known as the wrinkling phenomenon. Numerical simulations show that the functions $\nu^\eps$ tend to oscillate quickly in the whole of the unstable set 
$$
\Sigma=cl(\{W> W^{**}\}).
$$ 
However, in this paper we only prove that the wrinkling phenomenon occurs in a subset of $\Sigma$ and we do not explore further if it can be proven analytically that when oscillations occur, they actually encompass the whole of $\Sigma$. 

We prove that oscillations only occur inside of $\Sigma$ by proving that $d(\nu^\eps,\Sigma)$ is uniformly lower-semicontinuous in $\eps$ (see Corollary~\ref{cor:lsc}), which allows us to derive a uniform $H^1_{loc}$ estimate away from the unstable set (see Proposition~\ref{prop:H1loc}). The degenerate diffusion at $\{\nu^\eps=0\}$ makes the control near zero very subtle. Only a careful study of the behavior of the solution near zero can rule out uncontrolled jumps (see proof of Theorem~\ref{thm:osc}).

It is the intention of this paper that the proofs make a clear connection between where the oscillations can occur and the tangent lines of the graph of $W$. In short, in the regions where the tangent lines do not cross the graph of $W$, the function cannot have large oscillations (see \eqref{eq:condUlambda}). In this way, the function $W^{**}$ appears naturally and does not seem to be only a mathematical artifact of $\Gamma$-convergence. 

As usual with the framework of \cite{sandier2004gamma}, \cite{serfaty2011gamma}, we have to make an assumption on the initial data being well prepared with respect to the energy, meaning that
$$
\lim_{\eps\to0^+} E^\eps[\nu_i^\eps]
=E^{**}[\nu_i].
$$
In our case, the well preparedness can be interpreted as the fact that the approximations $\nu^\eps_i$ stays away from $\Sigma$, so the convergence we prove only tells us that asymptotically the dynamic keeps it that way. With this assumption, we are missing how the wrinkling phenomenon is actually affecting the dynamic in the limit, which is a really interesting question on its own, but needs to be analyzed more carefully with other types of techniques.

The paper is organized as follows: The rest of this Section deals with motivation. Section~\ref{sec:notations} provides the definitions and hypothesis of the objects we work with, and introduces a suitable notion of solution to \eqref{eq:gradeps}. Section~\ref{sec:result} contains the statements of the main result of the convergence (see Theorem~\ref{thm:conv}) and the main auxiliary result of the lower semicontinuity of the size of the gradients (see Theorem~\ref{thm:lscgrad}). Section~\ref{sec:osc} presents and proves the result on where can oscillations occur (see Theorem~\ref{thm:osc}). Section~\ref{sec:H1} proves that away from $\Sigma$ the functions are in $H^1$ (see Proposition~\ref{prop:H1loc}). Section~\ref{sec:prooflscgrad} proves Theorem~\ref{thm:lscgrad}. Section~\ref{sec:proofthmconv} proves Theorem~\ref{thm:conv}. Appendix~\ref{ap:wasserstein} gives the necessary background of gradient flows in $\W$. Appendix~\ref{sec:lscG} finishes the proof of the existence of an appropriate solution to \eqref{eq:gradeps} and proves the lower semicontinuity of $\mathcal{G}^\eps$ (see Lemma~\ref{lem:lscG}).

\subsection{Motivation}

Our original motivation for studying \eqref{eq:gradeps} came from a model for biological aggregation introduced in \cite{topaz2006nonlocal} which we describe now:

We consider $\nu(x,t)$ a population density that moves with velocity $v(x,t)$, where $x$, $v\in \R^n$, $t\ge0$. Then, $\nu$ satisfies the standard conservation equation, with initial population $\nu_0$
\begin{equation}\label{eq:model}
\left\{\begin{array}{rl}
\partial_t\nu+\nabla\cdot (v\nu)=&0\\
\nu(x,0)=&\nu_0.\end{array}\right.
\end{equation}

The model assumes that the velocity depends only on properties of $\nu$ at the current time and can be written as the sum of an aggregation and a dispersal term:
\begin{equation}\label{eq:v}
v=v_a+v_d.
\end{equation}

For aggregation, a sensing mechanism that degrades over distance, is hypothesized on the organisms. In the simplest case, the sensing function associated with an individual at position $x$ is given by
$$
s(x)=\int_{\R^n} K(x-y)\nu(y)\;dy=K*\nu(x),
$$
where the kernel $K$ is typically radially symmetric, compactly supported and of unit mass. Individuals aggregate by climbing gradients of the sensing function, so that the attractive velocity is given by:
\begin{equation}\label{eq:va}
v_a=\nabla K*\nu(x).
\end{equation}

Dispersal is assumed to arise as an anti-crowding mechanism and operates over a much shorter length scale. It is considered to be local, go in the opposite direction of population gradients and increase with density. For example we can take the dispersive velocity given by:
\begin{equation}\label{eq:vd}
v_d=-\nu\nabla\nu 
\end{equation}
(more generally $v_d=-f(\nu)\nabla\nu$).

Combining \eqref{eq:model}, \eqref{eq:v}, \eqref{eq:va} and \eqref{eq:vd}, we obtain the equation
\begin{equation}\label{eq:K3}
\left\{\begin{array}{rl}
\partial_t\nu+\nabla\cdot (\nu(\nabla K*\nu-\nu\nabla\nu))=&0\\
\nu(x,0)=&\nu_0.\end{array}\right.
\end{equation}

Now, by re-scaling, we want to consider what happens to a large population as we zoom out, over a large period of time. We thus set
$$
\int_{\R^n} \nu_0 \;dx=\eps^{-n},
$$
for some $\eps\ll 1$ and we re-scale time and space as follows:
\begin{equation}\label{eq:alpha}
\nu^\eps(x,t)=\nu\left(\frac{x}{\eps},\frac{t}{\eps^2}\right),
\end{equation}
the scaling in $x$ is chosen such that $\int\nu_0^\eps=1$.
Using \eqref{eq:K3}, we obtain the following equation for $\nu^\eps$:
\begin{equation}\label{eq:nueps}
\left\{\begin{array}{rl}
\partial_t\nu^\eps+\nabla\cdot (\nu^\eps(\nabla K^\eps*\nu^\eps-\nu^\eps\nabla\nu^\eps))=&0\\
\nu^\eps(x,0)=&\nu^\eps_0,\end{array}\right.
\end{equation}
where $K^\eps=\frac{1}{\eps^n}K(\frac{x}{\eps})$ is an approximation of the $\delta$ measure.

Adding and subtracting $\nabla\cdot(\nu^\eps\nabla \nu^\eps)$, we can rewrite \eqref{eq:nueps} as

\begin{equation}\label{eq:keps}
\partial_t\nu^\eps+\nabla\cdot (\nu^\eps(\nabla \nu^\eps-\nu^\eps\nabla\nu^\eps+(\nabla K^\eps*\nu^\eps-\nabla \nu^\eps)))=0.
\end{equation}
Assuming $\nu^\eps$ to be smooth, and taking a Taylor expansion of $\nu^\eps$, we get that
\begin{equation}\label{eq:Delta}
K^\eps*\nu^\eps(x)- \nu^\eps(x)=\eps^2k_0 \Delta \nu^\eps(x)+\mathcal{O}(\eps^4),
\end{equation}
where 
$$
k_0=\int |x|^2K(x)\;dx.
$$
Replacing \eqref{eq:Delta} in \eqref{eq:keps}, disregarding the $\mathcal{O}(\eps^4)$ term, we finally obtain \eqref{eq:gradeps}:
$$
\left\{\begin{array}{rl}
\partial_t\nu^\eps+\nabla\cdot (\nu^\eps(-\nabla W'(\nu^\eps)+\eps^2k_0 \nabla\Delta \nu^\eps))=&0\\
\nu^\eps(x,0)=&\nu^\eps_0,\end{array}\right.
$$
where $W'(x)=\frac{x^2}{2}-x$.

The Cahn-Hilliard equation we are studying in this paper is thus an approximation of the non-local equation \eqref{eq:nueps}. Unfortunately, the techniques used in this paper to control the oscillations of solutions to \eqref{eq:gradeps} could not be generalized to deal with solutions of \eqref{eq:nueps}. The main issue being that the non-locality does not allow us to integrate exactly against the derivative of the solution. We are thus unable, at the present time, to fully describe the behavior of the solutions of \eqref{eq:nueps} as $\eps\to0$. The only result that carries through is a uniform in $\eps$, $L^\infty$ estimate for the solutions of \eqref{eq:keps}, which follows almost exactly as Lemma~\ref{prop:infty}.

\begin{remark}
It is worth noticing that \eqref{eq:keps} is the gradient flow of
$$
F^\eps[\nu]=\int \frac{\nu^3(x)}{6}-\frac{1}{2}K^\eps*\nu(x)\nu(x)\;dx,
$$
with respect to the metric induced by the $\mathcal{W}^2$ distance. By adding and subtracting $\frac{\nu^2(x)}{2}$, in the expression, we obtain, after some calculations,
\begin{equation}\label{eq:Feps}
F^\eps[\nu]=\int W(\nu)\;dx+\frac{1}{4}\int\int K^\eps(x-y)(\nu(x)-\nu(y))^2\;dxdy,
\end{equation}
with $W(x)=\frac{x^3}{6}-\frac{x^2}{2}$.

The semi-norm
$$
\frac{1}{4}\int\int K^\eps(x-y)(\nu(x)-\nu(y))^2\;dxdy,
$$
is, up to a constant, a smooth non-local approximation of
$$
\frac{\eps^2}{2}\int |\nabla \nu|^2,
$$
therefore \eqref{eq:Feps} can be considered as a smooth non-local approximation of \eqref{eq:Eeps1}.
\end{remark}
\begin{remark}
Different scalings of time in \eqref{eq:alpha} can be considered. The case of $\frac{t}{\eps}$ is related, in the limit $\eps\to0$, to motion by mean curvature (see \cite{le2008gamma}).
\end{remark}
\begin{remark}
A similar heuristic relationship between \eqref{eq:gradeps} and the non-local model in \cite{topaz2006nonlocal} has been drawn independently in \cite{bernoff2015biological}.
\end{remark}

\section{Notation and Assumptions}\label{sec:notations}
Throughout the paper, we always consider measures $\mu\in\mathcal{P}(\T)$
that are absolutely continuous with respect to the Lebesgue measure, we do not make any distinction between the measure and its density. 

Also, we use the term sequence loosely: it may denote family of measures labeled by the continuous parameter $\eps$. 

\subsection{Assumptions on $W$}

We assume that $W$ is in $C^2([0,\infty),[0,\infty))$; we denote by $W^{**}$ its convex envelope. We define the auxiliary function $Q$ such that 
\begin{equation}\label{eq:Q'}
Q'(y)=yW'(y)-W(y),
\end{equation} 
we use the notation with a prime, because its derivative is related with the second derivative of $W$, namely
\begin{equation}\label{eq:Q''}
Q''(y)=yW''(y).
\end{equation}
Moreover we assume that $W$ has the following properties:
\begin{itemize}
\item (H1) There exists a constant $C>0$ such that for every $y\in\R$ 
\begin{equation}
|Q'(y)|\le C(1+W(y)).
\end{equation}
and
\begin{equation}
|W'(y)|\le C(1+W(y)).
\end{equation}
\item (H2) $\lim_{y\to +\infty}Q'(y)=+\infty$
\item (H3) The unstable set $\Sigma=cl(\{W> W^{**}\}\cup\{0\})=\cup_{i=1}^p \Sigma_i$, where $p\in \N$ and $\Sigma_i=[a_i,b_i]$, with $a_{i+1}>b_i$. 

The first interval could be degenerate in the sense of $a_1=b_1=0$.
As the dynamics near zero will be special, we will distinguish a value
\begin{equation}\label{eq:m0}
m_0=
\left\{\begin{array}{lr}
b_1+1&\mbox{if $p=1$}\\
\frac{b_1+a_2}{2}&\mbox{if $p\ge2$}.
\end{array}\right.
\end{equation}

\begin{figure}[H]\label{fig:1}
\centering
\def\svgwidth{12cm}
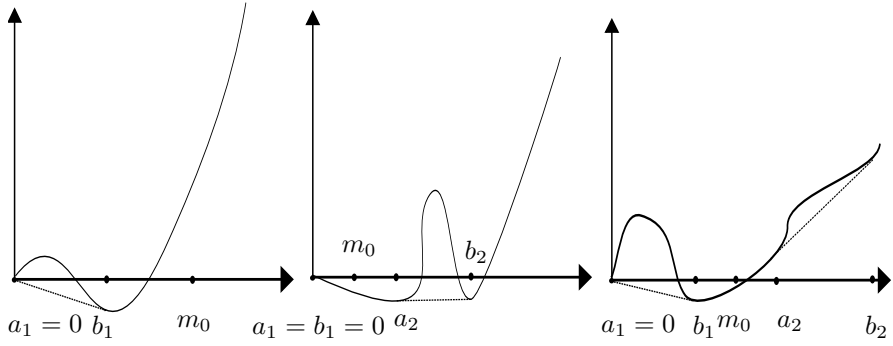
 \caption{From left to right: case $p=1$ with $a_1\ne b_1$, case $p=2$ with $a_1=b_1=0$ and case $p=2$ with $a_1\ne b_1$.}
\end{figure}

\item (H4) Given $K\subset \Sigma^c$ compact, $\inf_{A\in K}W''(A)>0$. Equivalently, $W''(p)>0$ for any $p\in\Sigma^c$.

\end{itemize}
\begin{remark}\label{rem:Q'=0}
Equation \eqref{eq:gradeps} and \eqref{eq:grad} are not affected by taking adding an affine function to $W$, so without loss of generality we will consider the case $W(0)=0$ and $W'(0)=0$.
\end{remark}
\subsection{Functionals $E^\eps$, $E^{**}$, $\mathcal{G}^\eps$, $|\nabla E^{**}|$}
For any $\eps>0$, we define 
$$
E^\eps:\mathcal{P}(\T)\to[0,+\infty]
$$
the functional
\begin{equation*}
E^\eps[\nu]=
\left\{\begin{array}{lr}
\int_\T\frac{\eps^2}{2}|\nu_x|^2+W(\nu)\; dx&\mbox{if }\nu\in H^1(\T)\\
+\infty&elsewhere.
\end{array}\right.
\end{equation*}

Formally, the subdifferential of $E^\eps$ at $\mu$, with respect to $\mathcal{W}^2$ is given by
$$
\partial_{\mathcal{W}^2} E^\eps [\mu]=\nabla (W'(\mu)-\eps^2\Delta\mu).
$$
and we have
$$
|\partial_{\mathcal{W}^2} E^\eps|(\mu)=\int_\T \mu|\nabla (W'(\mu)-\eps^2\Delta\mu)|^2\;dx.
$$
However, to our knowledge, unless $\mu$ is assumed to be strictly positive, nobody has proven that $E^\eps$ are actually sub-differentiable at $\mu$, no matter how regular $\mu$ is.

For this reason, we introduce a functional
$$
\mathcal{G}^\eps(\cdot):\mathcal{P}(\T)\to[0,+\infty],
$$
which will play the role of $|\partial E^\eps|$; we define it, using an auxiliary map $G^\eps$ and an auxiliary set $\mathcal{T}^\eps$. For $\mu\in H^1(\T)$, we define
\begin{equation}\label{eq:Geps}
G^\eps(\mu)=\mu W'(\mu)-W(\mu)+\frac{3\eps^2}{2}|\mu_x|^2-\eps^2(\mu\mu_x)_x
\end{equation}
(formally at least, we have $G^\eps(\mu)_x=\mu (W'(\mu)-\eps^2\mu_{xx})_x$)
and
$$
\mathcal{T}^\eps(\mu)=\{g\in L^2(\T):G^\eps(\mu)_x=\sqrt{\mu}g\}
$$
(possibly empty).
We then set
\begin{equation}\label{eq:calGeps}
\mathcal{G}^\eps(\mu)=
\left\{\begin{array}{lr}
\inf_{g\in\mathcal{T}^\eps(\mu)}\; ||g||_2 & \mbox{if }\mathcal{T}^\eps(\mu)\ne \varnothing,\\
+\infty&\mbox{otherwise}.
\end{array}\right.
\end{equation}
\begin{remark}\label{rem:relationGepssubdiff}
We always have the inequality
$$
\mathcal{G}^\eps(\mu)\le \int_\T \mu|(W'(\mu)-\eps^2\mu_{xx})_x|^2\;dx.
$$
Indeed, if the right hand side is infinite, there is nothing to prove, and if the right hand side is finite, then $\sqrt{\mu}(W'(\mu)-\eps^2 \mu_{xx})_x\in\mathcal{T}^\eps(\mu)$ and the inequality clearly holds.
\end{remark}
\begin{remark}\label{rem:cumbersome}
The idea, behind this cumbersome definition, is that 
$$
||G^\eps(\mu)_x||_1\le\mathcal{G}^\eps(\mu)
$$ 
and when $\mu$ is regular in $\{\mu>0\}$, then
$$
\int_{\mu>0} \mu|(W'(\mu)-\eps^2\mu_{xx})_x|^2\;dx\le\mathcal{G}(\mu).
$$
The fact that the integral is only on the set $\{\mu>0\}$ is a standard inconvenience in the thin film equation literature and is the source of many difficulties in proving the existence of curves of maximal slope of $E^\eps$. 
\end{remark}

We also define
$$
E^{**}:\mathcal{P}(\T)\to[0,+\infty]
$$
the functional
\begin{equation*}
E^{**}[\nu]=
\left\{\begin{array}{lr}
\int_\T W^{**}(\nu)\; dx&\mbox{if }\nu\in L^1(\T)\\
+\infty&elsewhere.
\end{array}\right.
\end{equation*}
$E^{**}$ is convex (see \cite{mccann1997convexity}) and its subdifferential is given by
$$
\partial_{\mathcal{W}^2} E^{**} [\mu]=\nabla W^{**\prime}(\mu).
$$

Therefore, we define the functional
\begin{equation}\label{eq:nablaE**}
|\nabla E^{**}|:\mathcal{P}(\T)\to[0,+\infty]
\end{equation}
by
\begin{equation*}
|\nabla E^{**}(\mu)|=
\left\{\begin{array}{lr}
\left(\int_\T \mu |(W^{**\prime}(\mu))_x|^2\; dx\right)^\frac{1}{2}&\mbox{if }(Q^{**\prime}(\mu))\in W^{1,1}(\T)\\
+\infty&elswhere,
\end{array}\right.
\end{equation*}
where $Q^{**\prime}(z)=zW^{**\prime}(z)-W^{**}(z)$. For more details, see Section 10.4.3 in \cite{ambrosio2006gradient}.
\begin{remark}
The subtlety of the doubling condition on $W^{**}$ is omitted, because we deal with measures that are bounded.
\end{remark}
\begin{remark}
Because $E^{**}$ is convex, we have that $|\nabla E^{**}|$ is a strong upper gradient. (See Definition~\ref{def:convex} and Definition~\ref{def:strongup})
\end{remark}

\subsection{Existence of $\nu^\eps$ and $\nu_0$}

Given $\eps>0$ and an initial condition $\nu^\eps_i$, such that
$$
E^\eps[\nu^\eps_i]<+\infty,
$$
we consider $\nu^\eps(x,t)$ solution of equation \eqref{eq:gradeps} given by the following proposition:
\begin{proposition}\label{prop:existenceeps}
Given $\nu_i^\eps\in\mathcal{P}_2(\T)$, such that $E^\eps[\nu_i^\eps]<\infty$, then there exists $\nu^\eps\in L^\infty((0,\infty);H^1(\T)) \cap L^2_{loc}((0,\infty);H^2(\T))\cap C^{1,4}_{loc}(\{\nu^\eps>0\})$ such that
\begin{equation}\label{eq:epsflow}
\int_0^\infty\int_\T \nu^\eps\phi_t\;dxdt-\int_0^\infty\int_\T (\eps^2\nu^\eps_{xx}-W'(\nu^\eps))(\nu^\eps\phi_x)_x\;dxdt=0,
\end{equation}
for every $\phi\in C^\infty_c((0,\infty)\times \T)$.

Moreover,
\begin{equation}\label{eq:epsgradflow}
E^\eps[\nu^\eps(t)]
+\frac{1}{2}\int_0^t\mathcal{G}(\nu^\eps)^2\; ds
+ \frac{1}{2}\int_0^t|\nu^{\eps\prime}|^2\;ds \le
E^\eps[\nu_i^\eps]\;\;\;\forall t>0,
\end{equation}
where $|\nu^{\eps\prime}|$ is the size of the metric derivative of $\nu^\eps$ with respect to $\W$ (See Definition~\ref{thm:metricder}).
\end{proposition}
\begin{remark}
We cannot claim that $\nu^\eps$ is a curve of maximal slope, as defined in \cite{ambrosio2006gradient}, since we do not prove that $\mathcal{G}^\eps$ is an upper gradient of $E^\eps$ (See Definition~\ref{def:strongup}).
\end{remark}
\begin{remark}\label{rem:regnueps}
From the inclusion $H^2\subset C^{1,\frac{1}{2}}$, we get that for almost every $t$, $\nu^\eps(t)\in C^{1,\frac{1}{2}}$.
\end{remark}
\begin{proof} The
existence of $\nu^\eps\in L^\infty((0,\infty);H^1(\T)) \cap L^2_{loc}((0,\infty);H^2(\T))$, that satisfies \eqref{eq:epsflow} is a particular case of Theorem 1 in \cite{lisini2012cahn}. More precisely, $\nu^\eps$ is constructed as any accumulation point of the discrete interpolation of the solutions of the appropriate JKO scheme. The fact that $\nu^\eps\in C^{1,4}_{loc}(\{\nu^\eps>0\})$ follows from Schauder estimates (see \cite{bernis1990higher}).

The proof of \eqref{eq:epsgradflow} which plays a central role in the proof of our main result is somewhat more technical, and is detailed in Appendix~\ref{sec:lscG}.
\end{proof}
As the functional $E^{**}$ is convex then, using Theorem~\ref{thm:lambdaconvexflows}, we denote by $\nu_0$ the unique gradient flow of $E^{**}$ emanating from $\nu_i$. Moreover, $\nu_0$ is also the unique distributional solution to
\begin{equation}\label{eq:flow}
\left\{\begin{array}{cl}
\partial_t\nu&=((W^{**\prime}(\nu))_x\nu)_x\\
\nu(0)&=\nu_i.
\end{array}\right.
\end{equation}
It can be characterized by either the Energy inequality, also known as the maximal slope condition
\begin{equation}
E^{**}[\nu(t)]
+\frac{1}{2}\int_0^t|\nabla E^{**}(\nu)|^2\;ds
+ \frac{1}{2}\int_0^t|\nu'|^2\;ds\le
E^{**}[\nu_i]\;\;\;\forall t>0,
\end{equation}
or the Energy equality
\begin{equation}\label{eq:gradflow}
E^{**}[\nu(t)]
+\frac{1}{2}\int_0^t|\nabla E^{**}(\nu)|^2\;ds
+ \frac{1}{2}\int_0^t|\nu'|^2\;ds=
E^{**}[\nu_i]\;\;\;\forall t>0.
\end{equation}

\section{Statement of the Result}\label{sec:result}
The main result of this paper is the following:
\begin{theorem}\label{thm:conv}
Let $\{\nu^\eps_i\}_\eps$, $\nu_i\in\mathcal{P}(\T)$ be such that
$$
E^\eps[\nu^\eps_i]<+\infty\;\;\;\mbox{and}\;\;\;E^{**}[\nu_i]<+\infty.
$$
Suppose that, 
\begin{equation}\label{eq:in}
\lim_{\eps\to0^+}\nu^\eps_i=\nu_i\;\;\;\mbox{in $\W$}
\end{equation}
and
\begin{equation}\label{eq:Ein}
\lim_{\eps\to0^+}E^{\eps}[\nu^\eps_i]=E^{**}[\nu_i].
\end{equation}
Then, for any $T>0$,
$$
\lim_{\eps\to0^+}\nu^{\eps}=\nu_0\;\;\;\mbox{in $C^0([0,T];\W)$},
$$
$$
\lim_{\eps\to0^+}\int_0^T\left(\mathcal{G}^\eps(\nu^\eps(t))-|\nabla E^{**}(\nu_0(t))| \right)^2dt=0
$$
and
$$
\lim_{\eps\to0^+}E^\eps[\nu^\eps(t)]=E^{**}[\nu_0(t)]\;\;\; \forall t\ge0,
$$
where $\nu^\eps$ is the solution of \eqref{eq:gradeps} given by Proposition~\ref{prop:existenceeps} with initial condition $\nu^\eps_i$ and $\nu_0$ is the unique Gradient flow of $E^{**}$ (solution of \eqref{eq:flow}) with initial condition $\nu_i$, with respect to the metric $\W$. 
\end{theorem}

As in \cite{mugnai2007allen}, \cite{le2008gamma}, \cite{bellettini2012convergence}, \cite{craig2015convergence}, \cite{colombo2012passing}  and \cite{arnrich2012passing} the key step in the proof of Theorem~\ref{thm:conv} is to prove the lower-semicontinuity in the convergence of $\mathcal{G}^\eps$ to $|\nabla E^{**}|$, more specifically we need to prove:

\begin{theorem}\label{thm:lscgrad}
Let $\{\rho^\eps\}_{\eps>0}$ be a sequence of functions in $\mathcal{P}(\T)$ such that $\rho^\eps\in C^1(\T)\cap C^4_{loc}\{\rho^\eps>0\}$, $\rho^\eps\to\rho_0$ in $\W$ and $\sup_\eps E^\eps(\rho^\eps)<\infty$, then
\begin{equation}\label{eq:lscgrad}
\liminf_{\eps\to0^+}\mathcal{G}^\eps(\rho^\eps)\geq |\nabla E^{**}|(\rho_0).
\end{equation}
\end{theorem}
The next two section is devoted to some preliminary compactness results, which are used in the proof of Theorem~\ref{thm:lscgrad} that can be found in Section~\ref{sec:thmlscgrad}. The proof of Theorem~\ref{thm:conv} can be found in Section~\ref{sec:proofthmconv}.

\section{Preliminary to the proof of Theorem~\ref{thm:lscgrad}}\label{sec:osc}
\subsection{Uniform $L^\infty$ estimate}
The first step is to prove a uniform $L^\infty$ estimate.

\begin{proposition}\label{prop:infty}
Let $\{\rho^\eps\}_{\eps>0}$ be a sequence of functions in $\mathcal{P}(\T)$ such that $\sup_\eps E^\eps[\rho^\eps]+\mathcal{G}^\eps(\rho^\eps)\le C$, then 
$$
\sup_\eps ||\rho^\eps||_\infty \le M<\infty.
$$
Moreover, up to a subsequence, 
$$
\rho^\eps\rightharpoonup\rho^0 \;\;\;\mbox{weak-$*$ } L^\infty.
$$
\end{proposition}
\begin{proof}
Consider $G^\eps(\rho^\eps)$ as in \eqref{eq:Geps}:
$$
G^\eps(\rho^\eps)=-\eps^2\rho^\eps\rho^\eps_{xx}+\frac{\eps^2}{2}(\rho_x^\eps)^2+Q'(\rho^\eps),
$$
with $Q'$ defined by \eqref{eq:Q'}.
By Remark~\ref{rem:cumbersome}, we know that
$$
||G^\eps(\rho^\eps)_x||_1\le \mathcal{G}^\eps(\rho^\eps)\le C.
$$
Moreover,
$$
\int_\T G^\eps(\rho^\eps)\;dx=\int_\T \frac{3}{2}\eps^2(\rho_x^\eps)^2+Q'(\rho^\eps)\;dx\le3E^\eps[\rho^\eps]+ D\int_\T (W(\rho^\eps)+1)\;dx,
$$
by (H1). Therefore, $G^\eps(\rho^\eps)$ is uniformly in $W^{1,1}(\T)$, which implies 
$$
\sup_\eps||G^\eps(\rho^\eps)||_\infty<\infty.
$$ 

Now, let's prove that $\rho^\eps$ is uniformly in $L^\infty$: take $x_0$, such that $\rho^\eps(x_0)=||\rho^\eps||_\infty$, then $\rho_x(x_0)=0$ and $\rho_{xx}(x_0)\le0$. We should note that, because $\mathcal{G}^\eps(\rho^\eps)\le C$, then $\rho^\eps\in C^{2,\frac{1}{2}}_{loc}(\{\rho^\eps>0\})$ and $\rho_{xx}(x_0)$ has a well defined value.

Now, we have the bound
$$
||G^\eps(\rho^\eps)||_\infty \ge G^\eps(\rho^\eps)(x_0)\ge Q'(\rho^\eps(x_0)),
$$ 
which, by assumption (H2), gives a bound for 
$$
\sup_\eps ||\rho^\eps||_\infty.
$$
\end{proof}

\begin{corollary}
Under the assumptions of Proposition~\ref{prop:infty}, $G^\eps(\rho^\eps)$ is bounded in $H^1(\T)$ uniformly in $\epsilon$. More precisely, we have the bound
$$
||G^\eps(\rho^\eps)_x||_{L^2}\le ||\rho^\eps||_{\infty}\mathcal{G}^\eps(\rho^\eps)\le C.
$$
Therefore, $G^\eps(\rho^\eps)\in C^{\frac{1}{2}}(\T)$ uniformly in $\eps$.
\end{corollary}

\subsection{Control of the oscillations in the good set}
The key in the proof of Theorem~\ref{thm:lscgrad} is to control the size of the oscillations of $\rho^\eps$ in the good sets. This will be given by the following Theorem:

\begin{theorem}\label{thm:osc}
Let $\{\rho^\eps\}_{\eps>0}$ be a sequence of functions in $\mathcal{P}(\T)$ such that $\sup_\eps E^\eps[\rho^\eps]+\mathcal{G}^\eps(\rho^\eps)\le C$, then, for any $L\ge0$ there exists $\delta(\eta,C)>0$, independent of $\eps$, such that for any $\eps<\eps_0(\eta,C,L)$ and any pair of sequences $x_\eps$, $y_\eps$ satisfying:
\begin{itemize}
\item $0<y_\eps-x_\eps<\delta$,
\item  $|\rho_x^\eps(x_\eps)|< L$ and $|\rho_x^\eps(y_\eps)|< L$,
\end{itemize}
we have either
$$
d(\rho^\eps(z),\Sigma)<\eta\;\;\;\; \forall z\in [x_\eps,y_\eps]\;\; $$
or
$$
|\rho^\eps(x_\eps)-\rho^\eps(y_\eps)|<\eta.
$$
\end{theorem}
Theorem~\ref{thm:osc} is similar to Lemma 5.5 in the paper by Belletini et al. \cite{bellettini2012convergence}. The main difference in the proof is that in \cite{bellettini2012convergence} they have control of the $H^1$ norm of
\begin{equation}\label{eq:eeps}
e^\eps(\rho^\eps)=W'(\rho^\eps)-\eps^2\rho^\eps_{xx},
\end{equation}
while we only have control on
\begin{equation}\label{eq:eepsint}
\int_\T |e^\eps(\rho^\eps)_x|^2\rho^\eps\;dx,
\end{equation}
which is degenerate near $\{\rho^\eps=0\}$.

Theorem~\ref{thm:osc} can be interpreted as a uniform lower semi-continuity for $d(\rho^\eps,\Sigma)$:
\begin{corollary}\label{cor:lsc}
Let $\{\rho^\eps\}_{\eps>0}$ be a sequence of functions in $\mathcal{P}(\T)$ such that $\sup_\eps E^\eps[\rho^\eps]+\mathcal{G}^\eps(\rho^\eps)\le C$ and that $\rho^\eps\to\rho_0$ in $\W$, then for x, any Lebesgue point of $\rho_0$, there exists $\eps_x$ and $\delta'=\delta'(d(\rho_0(x),\Sigma))$ such that for every $\eps<\eps_x$ and $y\in (x-\delta',x+\delta')$, we have
$$
d(\rho_0(y),\Sigma)> \frac{d(\rho_0(x),\Sigma)}{2}.
$$
Moreover, $\Omega:=\{\rho_0\notin \Sigma\}$ has an open representative.
\end{corollary}
\begin{proof}[Proof of Corollary~\ref{cor:lsc}]
We start with the following claim:

\vspace{0.3cm}

\textbf{Claim:}\textit{ For any $\beta>0$, we define $\delta(\beta)=\frac{\delta(\beta,C)}{4}$ and $\eps(\beta)=\eps(C,\beta,\frac{4M}{\delta})$ (Given by Theorem~\ref{thm:osc}). If for some $\eps\in(0,\eps_\beta)$, we have that $d(\rho^\eps(x),\Sigma)>2\beta$, then $d(\rho^\eps(y),\Sigma)>\beta$ for all $y\in (x-\delta_\beta,x+\delta_\beta)$.}

\vspace{0.3cm}

\textbf{Proof of the Claim:}
We take $\delta=\delta(\eta,C)$ given by Theorem~\ref{thm:osc}. Because we know that $\sup_\eps|\rho^\eps|\le M$, we have $osc_{(x+\frac{\delta}{4},x+\frac{\delta}{2})}\rho^\eps\le M$. Therefore, there exists $x_1\in (x+\frac{\delta}{4},x+\frac{\delta}{2})$ such that 
$$
|\rho^\eps_x(x_1)|\le\frac{4M}{\delta}.
$$
Similarly, there exists $x_2\in (x-\frac{\delta}{2},x\frac{\delta}{4})$ such that $|\rho^\eps_x(x_2)|\le\frac{4M}{\delta}$. 

If $\eps<\eps(C,\eta,\frac{4M}{\delta})$ given by Theorem~\ref{thm:osc}, then we can estimate the difference between the maximum and the minimum in $[x_2,x_1]$ (they are either a critical point or a boundary point). Therefore, we know that 
$$
osc_{(x_2,x_1)}\rho^\eps<\eta,
$$ 
by taking $\eta=\beta$ the \textbf{Claim} follows. 

\vspace{0.3cm}

Because $x$ is a Lebesgue point, for all $r$ small enough, we know that 
\begin{equation}\label{eq:rx}
\left|\frac{1}{2r}\int_{x-r}^{x+r}\rho_0(y)dy-\rho_0(x)\right|<\frac{d(\rho_0(x),\Sigma)}{6}
\end{equation}
We will fix $r_x<\frac{1}{2}\delta\left({\frac{d(\rho_0(x),\Sigma)}{3}}\right)$ such that \eqref{eq:rx} holds.

By Proposition~\ref{prop:infty}, we know that $\rho^\eps\to\rho_0$ weak-$*$ $L^\infty$; therefore, there exists $\eps_x$ such that for all $\eps<\eps_x$
$$
\left|\frac{1}{2r_x}\int_{x-r_x}^{x+r_x}\rho^\eps(y)dy-\frac{1}{2r_x}\int_{x-r_x}^{x+r_x}\rho_0(y)dy\right|<\frac{d(\rho_0(x),\Sigma)}{6}.
$$
So, if $\eps<\eps_x$, there exists $x_\eps\in(x-r_x,x+r_x)$, such that 
$$
|\rho^\eps(x_\eps)-\rho_0(x)|<\frac{d(\rho_0(x),\Sigma)}{3},
$$ 
hence
$$
d(\rho^\eps(x_\eps),\Sigma)>\frac{2}{3}d(\rho_0(x),\Sigma).
$$ 
By the \textbf{Claim}, if $\eps$ is small enough, it follows that $d(\rho^\eps(y),\Sigma)>\frac{d(\rho_0(x),\Sigma)}{3}$ for all $y\in(x_\eps-\delta\left(\frac{d(\rho_0(x),\Sigma)}{3}\right),x_\eps+\delta\left(\frac{d(\rho_0(x),\Sigma)}{3}\right))$. The result follows, because 
$$
(x-\frac{1}{2}\delta\left({\frac{d(\rho_0(x),\Sigma)}{3}}\right),x+\frac{1}{2}\delta\left({\frac{d(\rho_0(x),\Sigma)}{3}}\right))
\subset
(x_\eps-\delta\left({\frac{d(\rho_0(x),\Sigma)}{3}}\right),x_\eps+\delta\left({\frac{d(\rho_0(x),\Sigma)}{3}}\right)).
$$
\end{proof}

To prove Theorem~\ref{thm:osc} we start by looking at the behavior of $\rho^\eps$ on the set $\{\rho^\eps>h\}$ with $h>0$. This case follows exactly as Lemma 5.5 in \cite{bellettini2012convergence}; our main contribution here is to give a different proof in a simple case that makes the set $\Sigma=cl\{W>W^{**}\}$ appear more naturally.

\begin{lemma}\label{lem:osch}
Let $\{\rho^\eps\}_{\eps>0}$ be a sequence of functions in $\mathcal{P}(\T)$ such that $\sup_\eps E^\eps[\rho^\eps]+\mathcal{G}^\eps(\rho^\eps)\le C$, then for any $h>0$ and $L\ge0$ there exists $\delta(\eta,C,h)>0$, independent of $\eps$, such that for any $\eps<\eps_0(\eta,C,h,L)$ and any pair of sequences $x_\eps$, $y_\eps$ satisfying:
\begin{itemize}
\item $0<y_\eps-x_\eps<\delta$,
\item  $|\rho_x^\eps(x_\eps)|< L$ and $|\rho_x^\eps(y_\eps)|< L$,
\item $\rho^\eps(z)>h$ $\forall z\in [x_\eps,y_\eps]$ 
\end{itemize}
then we have either
$$
d(\rho^\eps(z),\Sigma)<\eta\;\;\;\; \forall z\in [x_\eps,y_\eps] 
$$
or
$$
|\rho^\eps(x_\eps)-\rho^\eps(y_\eps)|<\eta.
$$
\end{lemma}

\begin{proof}
We only give a sketch of the proof, which shows why the set $\Sigma$ appears naturally. For a complete proof see Lemma 5.5 in \cite{bellettini2012convergence}.

Since $\rho^\eps(z)>h$ for every $z\in(x_\eps,y_\eps)$ and $\eps$, then
$$
(e^\eps(\rho^\eps))_x=(W'(\rho^\eps)-\eps^2\rho^\eps_{xx})_x
$$
is bounded in $L^2(x_\eps,y_\eps)$ uniformly in $\eps$ (see \eqref{eq:eeps}, \eqref{eq:eepsint} and Remark~\ref{rem:cumbersome}). Moreover, we have
\begin{equation}\label{eq:eepsbound}
\int_{x_\eps}^{y_\eps} e^\eps(\rho^\eps)(z)dz=\int_{x_\eps}^{y_\eps}W'(\rho^\eps(z))-\eps^2\rho^\eps_{xx}(z)dz
\le C+2\eps^2 L.
\end{equation}
Sobolev's Embedding Theorem implies that $e^\eps(\rho^\eps)$ is also uniformly in bounded $C^{\frac{1}{2}}$.

Without loss of generality, we will assume that $\rho^\eps(x_\eps)\le\rho^\eps(y_\eps)$, and that $\rho_x^\eps(x_\eps)$, $\rho^\eps_x(y_\eps)\ge0$, if not we work with the closest minimum to $x_\eps$ and the closest maximum to $y_\eps$, inside the interval. We will also assume $\rho_{xx}^\eps(x_\eps)\ge0$ and $\rho_{xx}^\eps(y_\eps)\le0$;
if this condition is not satisfied, we can take
$$
\tilde{x_\eps}=\inf\{z:z\in(x_\eps,y_\eps)\cap\rho_{xx}^{\eps}(z)<0\cap\rho^\eps_x\ge0\}.
$$
Then, we obtain  
$$
|\rho^{\eps}(x_\eps)-\rho^{\eps}(\tilde{x_\eps})|\le L\delta.
$$
If $\rho^{\eps}(\tilde{x_\eps})_x>0$, then $\rho^{\eps}(\tilde{x_\eps})_{xx}>0$. If $\rho^{\eps}(\tilde{x_\eps})_x=0$ and $\rho^{\eps}(\tilde{x_\eps})_{xx}<0$, then $\rho^{\eps}(\tilde{x_\eps})$ is a maximum. If this happens, we consider $\tilde{z_\eps}$ the closest minimum to $\tilde{x_\eps}$, so that we get $\rho^\eps(\tilde{z_\eps})_{xx}\ge0$. We split the interval in three, $(x_\eps,\tilde{x_\eps})$, $(\tilde{x_\eps},\tilde{z_\eps})$ and $(\tilde{z_\eps},y_\eps)$, and we control each of the pieces separately.

If $\rho^{\eps}_{xx}(y_\eps)>0$, we can repeat the same arguments.

Multiplying $e^\eps(\rho^\eps)$ by $\rho^\eps_x(z)$ and integrating between $x_\eps$ and $y_\eps$ we get the following
\begin{equation*}
\begin{array}{rl}
\int_{x_\eps}^{y_\eps}e^\eps(\rho^\eps)\rho^\eps_x\;dz&=\frac{\eps^2}{2}(|\rho^\eps_{x}(y_\eps)|^2-|\rho^\eps_{x}(x_\eps)|^2)+W(\rho^\eps(y_\eps))-W(\rho^\eps(x_\eps))\\
&\le \frac{\eps^2}{2} L^2+W(\rho^\eps(y_\eps))-W(\rho^\eps(x_\eps)).
\end{array}
\end{equation*}
On the other hand, integrating by parts we also find
$$
\int_{x_\eps}^{y_\eps}e^\eps(\rho^\eps)\rho^\eps_x\;dz=-\int_{x_\eps}^{y_\eps}e^\eps_x(\rho^\eps)\rho^\eps\;dz+[e^\eps(\rho^\eps)\rho^{\eps}]_{x_\eps}^{y_\eps}.
$$
Because $e^\eps_x$ is uniformly in $L^2$ and $\rho^\eps$ uniformly in $L^\infty$, we have
$$
\left|\int_{x_\eps}^{y_\eps}e^\eps_x(\rho^\eps)\rho^\eps\;dz\right|\le C (y_\eps-x_\eps)^\frac{1}{2}\le C\delta^\frac{1}{2}.
$$
We decompose
\begin{equation}\label{eq:xepsyeps}
[e^\eps(\rho^\eps)\rho^{\eps}]_{x_\eps}^{y_\eps}=e^\eps(\rho^\eps)(x_\eps)[\rho^{\eps}(y_\eps)-\rho^{\eps}(x_\eps)]+[e^\eps(\rho^\eps)]_{x_\eps}^{y_\eps}\rho^{\eps}(y_\eps);
\end{equation}
using that $e^\eps$ is uniformly in $C^\frac{1}{2}$, we see that
$$
|[e^\eps(\rho^\eps)]_{x_\eps}^{y_\eps}\rho^{\eps}|\le C (y_\eps-x_\eps)^\frac{1}{2}\le C\delta^\frac{1}{2}.
$$
Combining the five equations above, we see that given any $\lambda>0$, we can choose $\eps$ and $\delta$ small enough, such that 
$$
W(\rho^\eps(x_\eps))+e^\eps(\rho^\eps)(x_\eps)(\rho^\eps(y_\eps)-\rho^\eps(x_\eps))+\lambda\ge W(\rho^\eps(y_\eps)).
$$
Using the assumption $\rho^\eps(y_\eps)\ge\rho^\eps(x_\eps)$, the definition of $e^\eps(\rho^\eps)$ and the fact that $\rho_{xx}^\eps(x_\eps)\ge0$, we obtain
\begin{equation}\label{eq:lambda1}
W(\rho^\eps(x_\eps))+W'(\rho^\eps(x_\eps))(\rho^\eps(y_\eps)-\rho^\eps(x_\eps))+\lambda\ge W(\rho^\eps(y_\eps)).
\end{equation}

Exchanging the roles of $x_\eps$ and $y_\eps$ in \eqref{eq:xepsyeps} and using the fact that $\rho_{xx}^\eps(y_\eps)\le0$, we obtain similarly
\begin{equation}\label{eq:lambda2}
W(\rho^\eps(y_\eps))-W'(\rho^\eps(y_\eps))(\rho^\eps(y_\eps)-\rho^\eps(x_\eps))+\lambda\ge W(\rho^\eps(x_\eps)).
\end{equation}

To use these conditions analytically, we define the sets
$$
U_\lambda^W(A)=\{B\in \R_+:W(A)+W'(A)(B-A)+\lambda\ge W(B)\},
$$
so conditions \eqref{eq:lambda1} and \eqref{eq:lambda2} can be reformulated as:
\begin{equation}\label{eq:condUlambda}
\rho^\eps(y_\eps)\in U_\lambda^W(\rho^\eps(x_\eps)) \mbox{ and } \rho^\eps(x_\eps)\in U_\lambda^W(\rho^\eps(y_\eps)).
\end{equation}

We now finish the proof under the extra assumption that $\Sigma\cap(h,\infty)$ contains only one interval:

By (H4), we know that for any fixed $\eta>0$, $W$ is uniformly convex in $\{p\in \R^+:d(p,\Sigma)\ge\eta\}$; therefore we can choose $\lambda_0$ such that for all $\lambda<\lambda_0$, we have
\begin{equation}\label{eq:Ulambda}
U^W_\lambda(A)\subset (A-\eta,A+\eta)\;\;\;\mbox{for all $A\in\{p\in \R^+:p>h\cap d(p,\Sigma)\ge\eta\}$}.
\end{equation}

If $d(\rho^\eps(x_\eps),\Sigma)>\eta$, then, with $A=\rho^\eps(x_\eps)$, \eqref{eq:condUlambda} and \eqref{eq:Ulambda} imply
$$
|\rho^\eps(x_\eps)-\rho^\eps(y_\eps)|<\eta.
$$
The same holds for $d(\rho^\eps(y_\eps),\Sigma)>\eta$.

On the other hand, if $d(\rho^\eps(x_\eps),\Sigma)<\eta$ and $d(\rho^\eps(x_\eps),\Sigma)<\eta$, we take 
$$
z=\argmax_{t\in[x_\eps,y_\eps]}d(\rho^\eps(t),\Sigma);
$$
if $d(\rho^\eps(z),\Sigma)<\eta$, we are done. If $d(\rho^\eps(z),\Sigma)>\eta$, because we assume that $\Sigma\cap(h,\infty)$ is an interval, we know that $\rho_x^\eps(z)=0$. Therefore, the intervals $(x_\eps,z)$ and $(z,y_\eps)$ satisfy the hypothesis of the Lemma. Then, arguing as before, because $d(\rho^\eps(z),\Sigma)>\eta$, we have
$$
|\rho^\eps(x_\eps)-\rho^\eps(z)|<\eta\mbox{ and }|\rho^\eps(y_\eps)-\rho^\eps(z)|<\eta.
$$
By our definition of $z$, we can conclude
$$
|\rho^\eps(x_\eps)-\rho^\eps(y_\eps)|<\eta,
$$
which proves the Lemma with the extra assumption of $\Sigma\cap(h,\infty)$ contains one interval.

A more convoluted argument, as the one in the proof of Theorem~\ref{thm:osc} (below), can be made for the cases when $\Sigma\cap(h,\infty)$ contains more than one interval. It is not included here, as a proof of this Lemma can already be found in \cite{bellettini2012convergence} and the ideas of the argument can be found in the proof below.
\end{proof}

We now turn to the proof of Theorem~\ref{thm:osc}:

\begin{proof}[Proof of Theorem~\ref{thm:osc}]
We prove Theorem~\ref{thm:osc} by contradiction. Due to Lemma~\ref{lem:osch}, we know that the theorem would be proven if we can prove that there is no $\eta_0>0$, such that there exist a sequence of $\eps_i\to0$ and sequences of points $\{x_i\}$, $\{y_i\}$, that satisfy for all $i$
\begin{itemize}
\item $|y_i-x_i|<\frac{1}{i}$
\item $\max(|\rho_x^{\eps_i}(x_i)|,|\rho_x^{\eps_i}(y_i)|)<L$
\item $\rho^{\eps_i}(x_i)\to0$
\item $\rho^{\eps_i}(y_i)>b_1+\eta_0$ (See (H3)).
\end{itemize}
So, let's assume that such $\eta_0$ exists and derive a contradiction.

Without loss of generality, we assume, as in the proof of Lemma~\ref{lem:osch}, that $\rho_x^{\eps_i}(x_i)\ge 0$, $\rho_x^{\eps_i}(y_i)\ge 0$, $\rho_{xx}^{\eps_i}(x_i)\ge 0$ and $\rho_{xx}^{\eps_i}(y_i)\le 0$.

From the proof of Proposition~\ref{prop:infty} we know that the function
$$
G^\eps(\rho^\eps)=-\eps^2\rho^\eps\rho_{xx}^\eps+\eps^2\rho_x^{\eps2}+Q'(\rho^\eps)
$$ 
is uniformly in $C^\frac{1}{2}$, using the fact that $\rho_{xx}^{\eps_i}(x_i)\ge0$ and $\rho_{xx}^{\eps_i}(y_i)\le 0$, we can conclude that
$$
Q'(\rho^{\eps_i}(y_i))\le Q'(\rho^{\eps_i}(x_i))+C\frac{1}{i^\frac{1}{2}}+\frac{{\eps_i}^2}{2}L^2.
$$
Moreover, since $\rho^\eps(x_i)\to0$, for every $\kappa>0$, there exists $i_0$ such that 
$$
Q'(\rho^{\eps_i}(x_i))+C\frac{1}{i^\frac{1}{2}}+\frac{{\eps_i}^2}{2}L^2
<\kappa,
$$
for all $i>i_0$. This implies that $Q'(\rho^{\eps_i}(y_i))<\kappa$, for all $i>i_0$.

The first observation is that $Q'(b_1)=b_1W'(b_1)-W(b_1)=0$. This is just saying that the tangent of $W$ at $b_1$ intersects the origin, which is satisfied because $b_1=\inf_{t>0}\{W(t)=W^{**}(t)\}$ (for a picture see Figure 1). 

The second observation is that by (H4), we have $\int_{s_1}^{s_2}tW''(t)=\int_{s_1}^{s_2}Q''(t)>0$ for $s_1,\;s_2\in(b_1,m_0)$, where $m_0$ is defined in \eqref{eq:m0}. We deduce that, for every $\eta>0$, there exists $\kappa_0$, such that if $A\in(b_1,m_0)$ and $Q'(A)<\kappa_0$, then $A-b_1<\eta$. Therefore, we get will get a contradiction, if we show that $\rho^{\eps_i}(y_i)<m_0$.

\vspace{0.3cm}
\textbf{Claim I:}\textit{If $i$ is large enough, then $\rho^{\eps_i}(y_i)<m_0$.}
\vspace{0.3cm}

Again, we will prove \textbf{Claim I} by contradiction; if $\rho^\eps(y_i)\ge m_0$, then there exists $z_0^i\in[x_i,y_i]$ such that $\rho^i(z_0^i)=m_0$. If we assume also that $\rho_{xx}^{\eps_i}(z_0^i)\le0$, then proceeding as above, we get
$$
0<Q'(m_0)\le Q'(\rho^{\eps_i}(x_i))+C\frac{1}{i^\frac{1}{2}}+\frac{{\eps_i}^2}{2}L^2,
$$
and taking $i$ large enough it yields our desired contradiction. Therefore, we want to prove that we can indeed assure that  $\rho_{xx}^{\eps_i}(z_0^i)\le0$, for $i$ large enough:

\vspace{0.3cm}

\textbf{Claim II:} \textit{Let
\begin{equation}\label{eq:z01}
z_0=\sup\{t\in(x_i,y_i):\rho^{\eps_i}(t)=m_0\}.
\end{equation}
If
\begin{equation}\label{eq:z0}
W(m_0)+W'(m_0)(\rho^{\eps_i}(y_i)-m_0)<W(\rho^{\eps_i}(y_i))-C(y_i-z_0)^\frac{1}{2}(\rho^{\eps_i}(y_i)-m_0),
\end{equation}
for some $C$ independent of $i$, then for all $i$ big enough
$$
\rho_{xx}^{\eps_i}(z_0)\le 0.
$$}
\vspace{0.3cm}

\textbf{Proof of Claim II:}

Due to the assumption on $z_0$, we know that $\rho^{\eps_i}>m_0$ in $(z_0,y_i)$. Therefore, we know that $e^{\eps_i}(t)=W'(\rho^{\eps_i}(t))-\eps^2\rho_{xx}^{\eps_i}(t)$ is uniformly in $H^1(z_0,y_i)$ (see \eqref{eq:eepsbound}). We perform the following calculation
\begin{equation*}
\begin{array}{rl}
\int_{z_0}^{y_i}e^{\eps_i}(t)\rho_x^{\eps_i}(t)dt
&=
W(\rho^{\eps_i}(y_i))-W(\rho^{\eps_i}(z_0))
-\frac{{\eps_i}^2}{2}|\rho_x^{\eps_i}(y_i)|^2
+\frac{{\eps_i}^2}{2}|\rho_x^{\eps_i}(z_0)|^2\\
&\ge
W(\rho^{\eps_i}(y_i))-W(m_0)-\frac{\eps_i^2}{2}L^2.
\end{array}
\end{equation*}

Using the same arguments used to derive \eqref{eq:lambda1} in the proof of Lemma~\ref{lem:osch}, we get
$$
e^{\eps_i}(z_0)(\rho^{\eps_i}(y_i)-\rho^{\eps_i}(z_0))+C(y_i-z_0)^\frac{1}{2}(\rho^{\eps_i}(y_i)-\rho^{\eps_i}(z_0))
\geq
W(\rho^{\eps_i}(y_i))-W(\rho^{\eps_i}(z_0))-\frac{\eps_i^2}{2}L^2.
$$

If $\rho^{\eps_i}_{xx}(z_0)\ge0$ , then $W'(\rho^{\eps_i}(z_0))\ge e^{\eps_i}(z_0)$, and so
$$
W(\rho^{\eps_i}(z_0))+W'(\rho^{\eps_i}(z_0))(\rho^{\eps_i}(y_i)-\rho^{\eps_i}(z_0))
\geq
W(\rho^{\eps_i}(y_i))-C(y_i-z_0)^\frac{1}{2}(\rho^{\eps_i}(y_i)-\rho^{\eps_i}(z_0))-\frac{\eps_i^2}{2}L^2.
$$
Since $\frac{\eps_i^2}{2}L^2\to0$, if $i$ is large enough this contradicts \eqref{eq:z0}, and thus proves \textbf{Claim II}.
\vspace{0.3cm}

To finish the proof of \textbf{Claim I}, we have to show that if $z_0$ defined by \eqref{eq:z01} exist (in particular, if $\rho^{\eps_i}(y_i)\ge m_0$), then \eqref{eq:z0} holds. First, we note that
$$
W(m_0)+W'(m_0)(t-m_0)<W(t) \;\; \forall t\neq m_0,
$$
due to (H4). Therefore, there exists $\kappa_0>0$ such that
$$
W(m_0)+W'(m_0)(t-m_0)<W(t)-\kappa_0\;\;\;\hbox{for every $t$ s.t. $Q'(t)<\frac{Q'(m_0)}{2}$}
$$
(the choice of $\frac{Q'(m_0)}{2}$ is arbitrary).

Recalling that
$$
\limsup_{i\to\infty} Q'(\rho^{\eps_i}(y_i))\le0,
$$
we can take $i$ large enough such that 
$$
Q'(\rho^{\eps_i}(y_i))<\frac{Q'(m_0)}{2}
$$
and
$$
\frac{C}{i^\frac{1}{2}}(\rho^{\eps_i}(y_i)-\rho^{\eps_i}(z_0))<\kappa_0,
$$ 
so we can conclude that \eqref{eq:z0} holds and \textbf{Claim II} yields 
$$
\rho_{xx}^{\eps_i}(z_0^i)\le0.
$$
This completes the proof of \textbf{Claim I} and of the Theorem.
\end{proof}

\section{$H^1$ estimate in the good set $\Omega$}\label{sec:H1}
We want to show that $\rho^\eps$ is bounded in $H^1_{loc}(\Omega)$ uniformly in $\eps$, with $\Omega=\{\rho_0\notin\Sigma\}$, in other words, that $\rho^\eps$ does not oscillate in the "good" limiting set. We start with the following proposition:

\begin{proposition}\label{prop:disttosigma}
Let $\{\rho^\eps\}_{\eps>0}$ be a sequence of functions in $\mathcal{P}(\T)$ such that $\sup_\eps E^\eps[\rho^\eps]+\mathcal{G}^\eps(\rho^\eps)\le C$ and $\rho^\eps\to\rho_0$ in $\W$, given $\phi\in\mathcal{D}(\Omega)$, there exists $\eps_0>0$ such that for every $\eps<\eps_0$, we have $W''(\rho^\eps)\ge\lambda_\phi$ and $\rho^\eps\ge\lambda_\phi$ in the support of $\phi$, for some constant $\lambda_\phi>0$ independent of $\eps$.
\end{proposition}

\begin{proof}
By assumption, if $x\in\Omega$, then $d(\rho_0(x),\Sigma(W))>0$. By Corollary~\ref{cor:lsc}, for any Lebesgue point $x\in \Omega$ there exists $\eps_x$ and $\delta_x$ such that for every $\eps<\eps_x$ and every $z\in(x-\delta_x,x+\delta_x)$
$$
d(\rho^\eps(z),\Sigma)>\frac{d(\rho_0(x),\Sigma)}{3}.
$$

Now, the family of intervals $\{(x-\delta_x, x+\delta_x)\}_{x\in \hat{\Omega}}$, where $\hat{\Omega}$ is the Lebesgue points of $\Omega$, is an open covering of the support of $\phi$, therefore by compactness there exists a finite sub-covering, which proves the proposition. 
\end{proof}
Using Proposition~\ref{prop:disttosigma} we can now prove:
 
 \begin{proposition}\label{prop:H1loc}
Let $\{\rho^\eps\}_{\eps>0}$ be a sequence of functions in $\mathcal{P}(\T)$ such that $\sup_\eps E^\eps[\rho^\eps]+\mathcal{G}^\eps(\rho^\eps)\le C$ and $\rho^\eps\to\rho_0$ in $\W$, for any $K\subset \Omega$ compact, there exists $C$ and $\eps_K>0$ such that
$$
\int_K|\rho_x^\eps|^2\;dx<C\;\;\;\;\;\forall \eps<\eps_K.
$$
Therefore, up to a subsequence, $\rho^\eps$ converges pointwise to $\rho_0$ a.e. in $\Omega$.
\end{proposition}
\begin{proof}
Take $\phi\in D(\Omega)$, with $\phi\ge\chi_K$. We start with the following computation,
$$
\int_\T \rho^\eps_x\partial_x[W'(\rho^\eps)-\eps^2\rho^\eps_{xx}]\phi\;dx
=\int_\T W''(\rho^\eps)|\rho_x^\eps|^2\phi\;dx +\eps^2\int_\T |\rho^\eps_{xx}|^2\phi\;dx+\eps^2\int_\T \rho^\eps_x\rho^\eps_{xx}
\phi_x\;dx,
$$
from which we deduce
$$
\begin{array}{rl}
\int_\T W''(\rho^\eps)|\rho^\eps_x|^2\phi &+\eps^2\int_\T |\rho^\eps_{xx}|^2\phi\;dx
=-\eps^2\int_\T \rho^\eps_x\rho^\eps_{xx}
\phi_x\;dx+\int_\T \rho^\eps_x\partial_x[W'(\rho^\eps)-\eps^2\rho^\eps_{xx}]\phi\;dx\\
&=\frac{\eps^2}{2}\int_\T |\rho^\eps_x|^2
\phi_x\;dx+\int_\T \rho^\eps_x\partial_{xx}[W'(\rho^\eps)-\eps^2\rho^\eps_{xx}]\phi\;dx\\
&\le C(\phi_{xx})\eps^2\int_\T |\rho_x^\eps|^2\;dx+\left(\int_\T|\rho_x^\eps|^2\phi\;dx \right)^\frac{1}{2}\left(\int_\T \phi|\partial_x[W'(\rho^\eps)-\eps^2\rho^\eps_{xx}]|^2\;dx\right)^\frac{1}{2}\\
&\le C\eps^2\int_\T |\rho_x^\eps|^2\;dx+\frac{\lambda_\phi}{2}\int_\T|\rho_x^\eps|^2\phi\;dx +C(\lambda_\phi)\int_\T \phi|\partial_x[W'(\rho^\eps)-\eps^2\rho^\eps_{xx}]|^2\;dx,
\end{array}
$$
with the constant $\lambda_\phi$ given by Proposition~\ref{prop:disttosigma}.
Therefore, we get:
$$
\begin{array}{rl}
\left(\inf_{x\in supp\{\phi\}} W''(\rho^\eps(x))-\frac{\lambda_\phi}{2}\right)\int_\T |\rho_x^\eps|^2\phi\;dx
\le& C\eps^2\int_\T |\rho_x^\eps|^2\;dx\\

&+C(\lambda_\phi)\int_{supp\{\phi\}} |\partial_x[W'(\rho^\eps)-\eps^2\rho^\eps_{xx}]|^2\;dx
\end{array}
$$
Using Proposition~\ref{prop:disttosigma} we can conclude that in the support of $\phi$ we have that $W''(\rho^\eps)>\lambda_\phi$ and that $\rho^\eps>\lambda_\phi$, for $\eps<\eps_0$, so we deduce
$$
\int_K|\rho_x^\eps|^2\;dx\le\int_\T |\rho_x^\eps|^2\phi\;dx\le C(\phi)E^\eps[\rho^\eps]+\frac{C(\lambda_\phi)}{\lambda_\phi}\mathcal{G}^\eps(\rho^\eps)\le C \;\;\;\;\forall\eps<\eps_0.
$$
\end{proof}

\begin{proposition}\label{prop:convQ'} Let $\{\rho^\eps\}_{\eps>0}$ be a sequence of functions in $\mathcal{P}(\T)$ such that $\sup_\eps E^\eps[\rho^\eps]+\mathcal{G}^\eps(\rho^\eps)\le C$ and $\rho^\eps\to\rho_0$ in $\W$, then
$$
(W'(\rho^\eps)-\eps^2\rho^\eps_{xx})_x\rho^\eps\to Q'(\rho_0)_x \;\;\;\mbox{in $D'(\Omega)$}.
$$
\end{proposition}
\begin{proof}
Fix $\phi\in D(\Omega)$, then using that $\rho W''(\rho)=Q''(\rho)$ with an integration by parts, we get
\begin{equation}\label{eq:prop6.3}
\int_\T (W'(\rho^\eps)_x-\eps^2\rho^\eps_{xxx})\rho^\eps\phi\;dx
=-\int_\T Q'(\rho^\eps)\phi_x\;dx +\eps^2\int_\T\rho^\eps_{xx}\rho_x^\eps\phi\;dx+\eps^2\int_\T\rho^\eps_{xx}\rho^\eps\phi_x\;dx
\end{equation}
The first term converges to what we are looking for
$$
\lim_{\eps\to0}-\int_\T Q'(\rho^\eps)\phi_x\;dx=-\int_\T Q'(\rho_0)\phi_x\;dx,
$$
by Lebesgue Dominated convergence and Proposition~\ref{prop:H1loc}.

It remains to show that the last two terms in \eqref{eq:prop6.3} go to zero. Integrating by parts again, we get
$$
\eps^2\int_\T\rho^\eps_{xx}\rho_x^\eps\phi\;dx+\eps^2\int_\T\rho^\eps_{xx}\rho^\eps\phi_x\;dx=-\frac{3}{2}\eps^2\int_\T|\rho_x^\eps|^2\phi_x\;dx-\eps^2\int_\T\rho^\eps_{x}\rho^\eps\phi_{xx}\;dx
$$
The first term goes to zero, by applying Proposition~\ref{prop:H1loc}. The second term can be re-written as
$$
\frac{1}{2}\eps^2\int_\T|\rho^\eps|^2\phi_{xxx}\;dx,
$$
which goes to zero, because $\rho^\eps$ is in $L^\infty$ uniformly.
\end{proof}

\section{Proof of Theorem~\ref{thm:lscgrad}}\label{sec:thmlscgrad}\label{sec:prooflscgrad}
\begin{proof}
To begin with, we assume that $\liminf \mathcal{G}^\eps(\rho^\eps)<\infty$, otherwise there is nothing to prove. Therefore, up to relabeling, we can consider $\rho^\eps$
such that 
$$
\sup_\eps E^\eps(\rho^\eps)+\mathcal{G}^\eps(\rho^\epsilon)<\infty.
$$
By Proposition~\ref{prop:infty}, we know that, up to subsequence, $\rho^\eps\rightharpoonup\rho_0$ weak-$*$ $L^\infty$; we define $\Omega=\{\rho_0\notin\Sigma\}$. We start with the following bound: using Proposition~\ref{prop:disttosigma}, for any $K\subset\Omega$ compact, we have, for all $\eps$ small enough
$$
\mathcal{G}^{\eps}(\rho^\eps)^2\ge\int_{\{\rho^\eps>0\}} |(W'(\rho^\eps)-\eps^2\rho_{xx}^\eps|^2\rho^\eps\;dx\ge\int_K |(W'(\rho^\eps)-\eps^2\rho_{xx}^\eps)_x|^2\rho^\eps\;dx.
$$
Furthermore, we have
$$
\int_K |(W'(\rho^\eps)-\eps^2\rho_{xx}^\eps)_x|^2\rho^\eps\;dx\ge2\int_\T (W'(\rho^\eps)-\eps^2\rho_{xx}^\eps)_x\phi \rho^\eps\;dx-\int_\T \phi^2\rho^\eps\;dx\;\;\;\;\mbox{for all $\phi\in\mathcal{D}(K)$},
$$
which implies
$$
\liminf_{\eps\to0} \mathcal{G}^{\eps}(\rho^\eps)^2\ge \liminf_{\eps\to 0}\left[2\int_\T (W'(\rho^\eps)-\eps^2\rho_{xx}^\eps)_x\phi \rho^\eps\;dx-\int_\T \phi^2 \rho^\eps\;dx\right]\;\;\;\;\mbox{for all $\phi\in\mathcal{D}(K)$}.
$$
By Proposition~\ref{prop:convQ'}, we deduce
$$
\liminf_{\eps\to0} \mathcal{G}^{\eps}(\rho^\eps)^2\ge\sup_{\phi\in D(K)}-2\int_\T Q'(\rho_0)\phi_x\;dx-\int_\T \phi^2\rho_0\;dx.
$$
By Proposition~\ref{prop:H1loc}, and the lower semi continuity of the $H^1$ seminorm we also know that $\rho_0$ is in $H^1_{loc}(\Omega)$, so we can integrate by parts
$$
\begin{array}{rl}
\lim_\eps \mathcal{G}^\eps(\rho^\eps)^2&\ge\sup_{\phi\in D(K)}2\int_\T Q'(\rho_0)_x\phi\;dx-\int_\T \phi^2d\rho_0\;dx\\
&=\sup_{\phi\in D(K)}2\int_\T W'(\rho_0)_x\rho_0\phi\;dx-\int_\T \phi^2d\rho_0\;dx\\
&=||W'(\rho_0)_x||^2_{L_{\rho_0}^2(K)}.\\
\end{array}
$$
Taking $K\to\Omega$ we obtain
$$
\lim_\eps |\mathcal{G}^\eps|^2(\rho^\eps)
\ge
||W'(\rho_0)_x||^2_{L_{\rho_0}^2(\Omega)}.
$$

The rest of the proof is devoted to proving
$$
||W'(\rho_0)_x||^2_{L_{\rho_0}^2(\Omega)}=|\nabla E^{**}|(\rho_0).
$$

First, to have $||W^{**\prime}(\rho_0)_x||^2_{L^2_{\rho_0}(\T)}=|\nabla E^{**}|(\rho_0)$, we need to prove that $Q^{**\prime}\in W^{1,1}$ (see \eqref{eq:calGeps}). We prove this by proving that
$$
Q^{**\prime}(\rho_0)_x=Q^{**\prime}(\rho_0)_x \mathbbm{1}_\Omega\;\;\;\mbox{in $\mathcal{D}'(\T)$}.
$$

Since $\rho^\eps$ is continuous, if $\rho^\eps(x)\in\Sigma_i$ and $\rho^\eps(y)\in\Sigma_j$, there exists $z\in(x,y)$ such that $d(\rho^\eps(z),\Sigma)\ge\inf_{i\ne j}\frac{d(\Sigma_i,\Sigma_j)}{2}$. By Corollary~\ref{cor:lsc}, we know that $d(\rho^\eps,\Sigma)$ is uniformly lower semi continuous, therefore there exists $\delta_0$, independent of $\eps$, such that $d(\rho^\eps(t)),\Sigma)>0$, for any $t\in (z_0-\delta_0,z_0+\delta_0)$, then $|x-y|>2\delta_0$. This implies that the sets $C_i=\{\rho_0\in \Sigma_i\}$ are at a non zero distance from each other.

We define, as an auxiliary function, $w$ in $\Sigma$ by 
$$
w(x)=Q^{**\prime}(\Sigma_i)\;\;\mbox{if $x\in C_i$},
$$
and we extend it to the whole of $\T$ by linear interpolation. Since the sets $C_i$ are separated, the function $w$ is Lipschitz. Moreover, $Q^{**\prime}(\rho_0)=w$ in $\Omega^c$, then for every $\phi\in\mathcal{D}(\T)$
$$
\int_\T (Q^{**\prime}(\rho_0)-w)\phi_x\;dx=\int_\Omega (Q^{**\prime}(\rho_0)-w)\phi_x\;dx.
$$
Integrating by parts we have no boundary term, and so
$$
\int_\T (Q^{**\prime}(\rho_0)-w)\phi_x\;dx=-\int_\Omega (Q^{**\prime}(\rho_0)-w)_x\phi\;dx.
$$
Because $w$ is Lipschitz and $w_x=0$ in $\Sigma$, then
$$
\int_\Omega w_x\phi\;dx=\int_\T w_x\phi\;dx=-\int_\T w\phi_x\;dx.
$$
Therefore, we obtain that
$$
\int_\T Q^{**\prime}(\rho_0)\phi_x\;dx=-\int_\Omega Q^{**\prime}(\rho_0)_x\phi\;dx,
$$
for every $\phi\in\mathcal{D}(\T)$.

Similarly, we can prove that $W^{**\prime}(\rho_0)_x=W'(\rho_0)_x\mathbbm{1}_\Omega$, so we obtain the desired equality:
$$
||W'(\rho_0)_x||^2_{L_{\rho_0}^2(\Omega)}=|\nabla E^{**}|(\rho_0)=||W^{**\prime}(\rho_0)_x||^2_{L^2_{\rho_0}(\T)}).
$$
\end{proof}

\section{Proof of Theorem~\ref{thm:conv}}\label{sec:proofthmconv}
\begin{proof}
To be able to apply the framework developed by Sandier-Serfaty \cite{sandier2004gamma}, we have to prove the compactness of the family $\nu^\eps$ with respect to time. 

Because the diameter of $\T$ is finite we know that the diameter of $\W$ is also finite, then
$$
\nu^\eps\in L^\infty([0,T];\W).
$$
By the energy inequality \eqref{eq:epsgradflow}, we know that
$$
\int_0^T|\nu^{\eps\prime}(t)|^2\;dt
$$ 
is uniformly bounded and therefore we know that
$$
\mbox{$\nu^\eps$ is uniformly bounded in $H^1((0,T);\W)$}.
$$
By \cite{lisini2007characterization}, we deduce that $\nu^\eps$ is precompact in $L^2([0,T];\W)$, so up to a subsequence 
$$
\nu^\eps\to \mu\;\;\;\mbox{in $L^2([0,T];\W)$}.
$$
Also, 
$$
\int_0^T|\nu^{\eps\prime}(t)|^2\;dt=\sup_{h\in(0,T)}\int_0^{T-h} \frac{d_2(\nu^\eps(t),\nu^\eps(t+h))}{h}\;dt
$$
is lower-semicontinuous with respect to the convergence in $L^2([0,T];\W)$, hence
\begin{equation}\label{eq:lsctime}
\liminf_{\eps\to0}\int_0^T|\nu^{\eps\prime}(s)|^2ds\ge\int_0^T|\mu'(s)|^2ds.
\end{equation}
Furthermore, by Arsela-Ascoli, we also know that up to a further subsequence,
$$
\nu^\eps\to\mu\;\;\;\mbox{in $C^0([0,T];\W)$.}
$$ 
In particular, 
$$
\nu^\eps_i\to\nu_i=\mu(0).
$$

Now, we only have to follow the proof in \cite{sandier2004gamma} and obtain that $\mu$ is the gradient flow of $E^{**}$ with initial condition $\nu_i$:

By equation \eqref{eq:epsgradflow}, we know that
$$
E^\eps[\nu_i^\eps]-E^\eps[\nu^\eps(t)]\ge
\frac{1}{2}\int_0^t\mathcal{G}^\eps(\nu^\eps)^2\;ds
+ \frac{1}{2}\int_0^t|\nu^{\eps\prime}|^2\;ds
$$
Taking the limit $\eps\to0$, using Fatou's Lemma, Theorem~\ref{thm:lscgrad} and \eqref{eq:lsctime} we get that
\begin{equation}\label{eq:enereps}
\liminf_{\eps\to0^+}(E^\eps[\nu_i^\eps]-E^\eps[\nu^\eps(t)])\ge
\frac{1}{2}\int_0^t|\nabla E^{**}(\mu)|^2\;ds
+ \frac{1}{2}\int_0^t|\mu'|^2\;ds.
\end{equation}
By Young's inequality, we know that
\begin{equation}\label{eq:ener**}
\frac{1}{2}\int_0^t|\nabla E^{**}(\mu)|^2\;ds
+ \frac{1}{2}\int_0^t|\mu'|^2\;ds\ge\int_0^t|\nabla E^{**}(\mu)||\mu'|\;ds.
\end{equation}
Because $E^{**}$ is convex with respect to the geodesics in $\W$, we can apply Theorem~\ref{thm:convexstrongupgrad} to obtain
\begin{equation}\label{eq:fundcalc}
\int_0^t|\nabla E^{**}(\mu)||\mu'|\;ds\ge E^{**}[\mu(0)]-E^{**}[\mu(t)].
\end{equation}

Since $\lim_\eps E^\eps[\nu_i^\eps]=E^{**}[\nu_i]=E^{**}[\mu(0)]$ by the well preparadness assumption, \eqref{eq:enereps}, \eqref{eq:ener**} \eqref{eq:fundcalc} imply
$$
\limsup E^\eps[\nu^\eps(t)]\le E^{**}[\mu(t)].
$$
The reverse inequality comes from the $\Gamma$-convergence of $E^\eps$ to $E^{**}$, so we have proven that
$$
\lim E^\eps[\nu^\eps(t)]= E^{**}[\mu(t)],
$$
and the inequalities \eqref{eq:enereps}, \eqref{eq:ener**} \eqref{eq:fundcalc} are in fact equalities. In particular \eqref{eq:enereps} yields
$$
E^{**}[\nu_i]-E^{**}[\mu(t)]=
\frac{1}{2}\int_0^t|\nabla E^{**}(\mu)|^2\;ds
+ \frac{1}{2}\int_0^t|\mu^{\prime}|^2\;ds
$$
for all $t>0$. Finally, by Theorem~\ref{thm:lambdaconvexflows}, we deduce that $\nu_0=\mu$.
\end{proof}

\appendix

\section{Gradient Flows in $\W$}\label{ap:wasserstein}
\subsection{General Metric Spaces}
We briefly review some important Definitions and Theorems of "Gradient flows: in metric spaces and in the space of probability measures" by Ambrosio, Gigli and Savare \cite{ambrosio2006gradient}. We start with some notions defined for a general complete metric space $(X,d)$, which we later analyze in the $\W=(\mathcal{P}(\T),d_2)$ case. We begin with the notion of an absolutely continuous curve
\begin{definition}
Let $v:(0,1)\to X$ be a curve, we say that $v\in AC^p(a,b;X)$ with $p\in[1,\infty)$, if there exists $m\in L^p(a,b)$ such that
\begin{equation}\label{eq:AC}
d(v(s),v(t))\le \int_s^t m(r)\;dr\;\;\;\; \forall a<s<t<b.
\end{equation}
If $p=1$, we suppress the superscript and just denote it by $AC$.
\end{definition}
Absolutely continuity is enough to define the size of a derivative at almost every point, this is the subject of the next theorem
\begin{theorem}\label{thm:metricder}
Let $v\in AC^p(a,b;X)$, then the limit
$$
|v'(t)|:=\lim_{h\to0}\frac{d(v(t+h)-v(t))}{h}
$$
exists a.e. in $(a,b)$ and $|v'|\in L^p(a,b)$. Moreover, it is minimal in the sense that it holds \eqref{eq:AC}, and if
$$
d(v(s),v(t))\le \int_s^t m(r)\;dr\;\;\;\; \forall a<s<t<b,
$$
then $|v'(t)|\le m(t)$ a.e. in $(a,b)$.
\end{theorem}

Now that we have the concept of the size of the derivative of a curve, we can give a notion the size of gradients for functionals defined in $X$. From now on, $\phi$ is a lower semi-continuous real-valued function on $X$.
\begin{definition}~\label{def:strongup}
A function $g:X\to[0,+\infty]$ is a strong upper gradient for $\phi$ if for any $v\in AC(a,b;X)$, the function $g\circ v$ is borel and
$$
|\phi(v(t))-\phi(v(s))|\le\int_s^t g\circ v(r)|v'(r)|\;dr\;\;\;\;\forall a<s<t<b.
$$
In particular, if $g\circ v(r)|v'(r)|\in L^1(a,b)$, then $\phi\circ v$ is absolutely continuous and
$$
|(\phi\circ v)'(t)|\le g\circ v(r)|v'(r)|\;\;\;\mbox{a.e. in (a,b)}.
$$
\end{definition}
The most natural candidate to satisfy the definition above is the slope of $\phi$.
\begin{definition}
The slope at $\phi$ at $v$ is defined by
$$
|\partial \phi(v)|:=\limsup_{w\to v}\frac{(\phi(w)-\phi(v))^+}{d(w,v)}.
$$
\end{definition}
To be able to relate the two definitions we need to consider a more restrictive set of functionals, for instance $\lambda$-convex functionals.
\begin{definition}\label{def:convex}
Given $\lambda\in\R$, we say that $\phi$ is $\lambda$-convex with respect to the geodesics, if for every $\gamma_t:[0,1]\to X$ constant speed geodesic, we have that
$$
\phi(\gamma_t)\le (1-t)\phi(\gamma_0)+t\phi(\gamma_1)-\frac{1}{2}\lambda t(1-t) (d(\gamma_0,\gamma_1))^2.
$$
\end{definition}

With this definition we can write the following Theorem

\begin{theorem}\label{thm:convexstrongupgrad}
Suppose that $\phi$ is $\lambda$ convex with respect to the geodesics, then $|\partial \phi|$ is a strong upper gradient.
\end{theorem}
\begin{proof}
See Corollary 2.4.10 in \cite{ambrosio2006gradient}
\end{proof}

Now, we are ready to define the curves of maximal slope for $\lambda$-convex functionals,

\begin{definition}
We say that the locally absolutely continuous map $u:(a,b)\to X$ is a curve of maximal slope of $\phi$ with respect to its upper gradient $|\partial \phi|$ if
\begin{equation}\label{eq:maxslope}
\phi(u(t))-\phi(u(s))\ge\int_s^t\frac{|u'(r)|^2}{2}+\frac{|\partial \phi(u(r))|^2}{2}\;dr.
\end{equation}
\end{definition}
\begin{remark}
If $(X,d)$ is a Hilbert space, and $\phi$ is $\lambda$-convex, then $|\partial\phi(v)|$ is actually the norm of the minimal selection in the sub-differential at $v$. Moreover, $u(\cdot)$ is a curve of maximal slope, if and only if, $u(\cdot)$ is a gradient flow. This follows from an application of the Cauchy-Schwartz and Young's inequality. 
\end{remark}

\subsection{Wasserstein metric}
We start with some auxiliary definitions to be able to define the $L^2$-Wasserstein distance, $d_2$.

\begin{definition}
Given $\mu$, $\nu\in\mathcal{P}(\T)$, we call $\pi\in\mathcal{P}(\T\times\T)$ a transference plan if
$$
\pi(A\times\T)=\mu(A)\;\;\;\mbox{and}\;\;\;\pi(\T\times A)=\nu(A),
$$
for every Borel set $A$. 

We denote the set of transference plans from $\mu$ to $\nu$ as $\Pi(\mu,\nu)$.
\end{definition}
\begin{remark}
$\mu\times\nu\in\Pi(\mu,\nu)$.
\end{remark}

\begin{definition}
Given $\mu$, $\nu\in\mathcal{P}(\T)$, we define their $L^2$-Wasserstein distance as
$$
d_2^2(\mu,\nu)=\inf_{\pi\in\Pi(\mu,\nu)}\left\{\int_{\T\times \T}|x-y|^2\;d\pi(x,y)\right\}.
$$
\end{definition}
This distance has been extensively studied in the literature and we recommend \cite{villani2008optimal} and \cite{ambrosio2013user}, which also contains a pedagogical introduction to the gradient flow theory. In this work, we are mostly interested on its differential structure.
\begin{theorem}\label{thm:abscontinuity}
Let the curve $\mu_t:I\to\mathcal{P(\T)}$ be absolutely continuous with respect to $d_2$ and let $|\mu'|\in L^1(I)$ be its metric derivative, then there exists a Borel vector field $v$ such that
$$
||v(\cdot,t)||_{L^2_{\mu_t}}\le|\mu'(t)|\;\;\; \mbox{a.e. $t\in I$}
$$
and the continuity equation 
\begin{equation}\label{eq:continuity}
\partial_t\mu_t+\nabla\cdot(v(\cdot,t)\mu_t)=0
\end{equation}
is solved in the sense of distributions.

Conversely, if $\mu_t:I\to\mathcal{P(\T)}$ is continuous with respect to $d_2$ and satisfies the continuity equation \eqref{eq:continuity} for some Borel velocity field $v$ with $||v(\cdot,t)||_{L^2_{\mu_t}}\in L^1(I)$, then $\mu_t$ is absolutely continuous and $|\mu'(t)|\le||v(\cdot,t)||_{L^2_{\mu_t}}$ a.e. $t\in I$.
\end{theorem}
\begin{proof}
See Theorem 8.3.1. \cite{ambrosio2006gradient}
\end{proof}
\begin{remark}
We are missing an extra condition to uniquely determine the vector field $v$, as we could always perturb it by a field $w$ such that $\nabla\cdot(w\mu_t)=0$, without changing the continuity equation \eqref{eq:continuity}.
\end{remark}
\begin{definition}
Let $\mu\in\mathcal{P}(\T)$, we define
$$
Tan_\mu\mathcal{P}(\T)=cl(\{\nabla\phi:\phi\in C^\infty(\T)\}),
$$
where $cl$ denotes the closure with respect to the $L^2_\mu$ topology.
\end{definition}
\begin{theorem}\label{thm:abscontinuitytan}
Let $\mu_t:I\to\mathcal{P}(\T)$ be an absolutely continuous curve and let $v$ be such that the continuity equation \eqref{eq:continuity} is satisfied. Then,
$|\mu'(t)|=||v(\cdot,t)||_{L^2_{\mu_t}}$ a.e. $t\in I$, if and only if, $v\in Tan_{\mu_t}\mathcal{P}(\T)$ a.e. $t\in I$. 
Moreover, the vector field $v$ is a.e. uniquely determined by $|\mu'(t)|=||v(\cdot,t)||_{L^2_{\mu_t}}$.
\end{theorem}
\begin{proof}
See Theorem 8.3.1. \cite{ambrosio2006gradient}. 
\end{proof}
Using the inner product structure in $L^2_\mu$, we are able to define the subdifferential of a $\lambda$-convex functional
\begin{definition}\label{def:strongsub}
We say that $\zeta\in L^2_\mu(\T)$ is a strong subdifferential of $\phi$ at $\mu$, denoted by $\partial \phi(\mu)$, if
$$
\phi(H\#\mu)-\phi(\mu)\le\int_\T <\zeta(x),H(x)-x>\;d\mu(x)+o(||H-I||_{L^2_\mu(\T)}),
$$
where $H$ is a Borel vector field and the push-forward $H\#\mu$ is defined by the condition $H\#\mu(A)=\mu(H^{-1}(A))$ for every Borel set $A$.
\end{definition}
We would like to characterize the strong subdifferentials of functionals like $E^\eps$, we consider
$$
\mathcal{F}[\mu]=\left\{\begin{array}{lr}
\int_\T F(x,\rho(x),\nabla\rho(x))dx&\mbox{if $\mu=\rho\; d\mathcal{L}$ and $\rho\in C^1(\T)$}\\
+\infty&\mbox{otherwise,}
\end{array}\right.
$$
where $d\mathcal{L}$ is the Lebesgue measure in $\T$. We denote $(x,z,p)\in\T\times\R\times\R$ the variables of F. To simplify, we ask that $F\in C^2$ and that $F(x,0,p)=0$ for every $x$ and $p$.
\begin{lemma}\label{lem:subdiff}
If $\mu=\rho\;d\mathcal{L}\in\mathcal{P}(T)$, with $\rho\in C^1$, satisfies $\mathcal{F}[\mu]<\infty$, then any strong subdifferential of $\mathcal{F}$ at $\mu$ is $\mu$-a.e. equal to
\begin{equation}
\nabla\frac{\delta\mathcal{F}}{\delta \rho}=\nabla(F_z(x,\rho(x),\nabla\rho(x))-\nabla\cdot F_p(x,\rho(x),\nabla\rho(x)).
\end{equation}
\end{lemma}
\begin{proof}
See Lemma 10.4.1. in \cite{ambrosio2006gradient}.
\end{proof}

Now we can define the notion of gradient flow for a functional $\phi$
\begin{definition}
We say that a map $\mu_t\in AC^2((0,\infty),\mathcal{P}(\T))$ is a solution to the gradient flow equation, if the vector field $v$ from Theorem~\ref{thm:abscontinuitytan} satisfies
$$
v(\cdot,t)\in\partial \phi(\mu_t)\;\;\forall t>0.
$$
\end{definition}
Now, in the $\lambda$-convex case, we can make the connection between gradient flows and curves of maximal slope
\begin{theorem}\label{thm:lambdaconvexflows}
If $\phi$ is $\lambda$-convex, then $\mu_t$ is a curve of maximal slope with respect to $|\partial\phi|$, if and only if, $\mu_t$ is a gradient flow and $\phi(\mu_t)$ is equal a.e. to a function of bounded variation.

Moreover, given two gradient flows $\mu_t^1$ and $\mu_t^2$, such that $\mu_t^1\to\mu_1$ and $\mu_t^2\to\mu_2$ as $t\to0$, then
$$
d_2(\mu_t^1,\mu_t^2)\le e^{-\lambda t}d_2(\mu^1,\mu^2).
$$
In particular, there is a unique gradient flow $\mu_t$ with initial condition $\mu_0$ and it satisfies the maximal slope condition \eqref{eq:maxslope} with equality.
\end{theorem}
\begin{proof}
See Theorem 11.1.3 and Theorem 11.1.4 in \cite{ambrosio2006gradient}.
\end{proof}

\section{Lower semi-continuity of $\mathcal{G}^\eps$}\label{sec:lscG}
In this section, we complete the proof of Proposition~\ref{prop:existenceeps}
\begin{proof}[Proof of Proposition~\ref{prop:existenceeps}(continuation)]

To lighten the notations, we give a proof for the case $\eps=1$, and drop the $\eps$ dependence.

The existence of solutions to \eqref{eq:epsflow} is proved by considering the uniform JKO scheme starting from $\nu_i$, with step $\tau>0$. Namely, we define inductively
$$
\mu^0_{\tau}=\nu^1_i,\;\; \mu_{\tau}^{n+1}=\argmin_{\rho\in \mathcal{P}(\T)}\{d_2^2(\mu^n_{\tau},\rho)+2\tau E(\rho)\}.
$$
By Lemma 1.2 in \cite{otto1998lubrication}, we know that $\mu_\tau^{n+1}$ solves
$$
\frac{\mu_\tau^{n+1}-\mu_\tau^{n}}{\tau}=(\mu_\tau^{n+1}(W'(\mu_\tau^{n+1})-\mu_{\tau xx}^{n+1})_x)_x,
$$
in the sense of distributions. Schauder's estimates yield $\mu_\tau^{n+1}\in C_{loc}^4(\{\mu_\tau^{n+1}>0\})$.

The existence of a solution to \eqref{eq:epsflow} follows from Theorem 1 in \cite{lisini2012cahn}, proven by defining $\nu$ as any accumulation point of the piecewise constant interpolation of $\{\mu_\tau^n\}_{n=0}^\infty$. Note that $\nu$ may not be unique, so we fix such a $\nu$ and a corresponding sequence of $\tau\to0$, for which the constant interpolation of $\{\mu_\tau^n\}_{n=0}^\infty$ converges to $\nu$.

Subsequently, we define the De Giorgi variational interpolation by
$$
\overline{\mu}_\tau(t)=\argmin_{\rho\in \mathcal{P}(\T)} \{d_2^2(\mu^n_{\tau},\rho)+2(t-(n-1)\tau) E(\rho)\}\;\;\mbox{when}\; t\in((n-1)\tau,n\tau).
$$
By Lemma 3.1.3 and 3.2.2 in \cite{ambrosio2006gradient}, we know that $E$ is strongly subdifferentiable at $\overline{\mu}_\tau(t)$ for every $t>0$ (see Definition~\ref{def:strongsub}), that $\overline{\mu}_\tau(t)\to\nu(t)$ for all $t\ge0$, and that for every $n\in\N$,
$$
E(\mu^n_\tau)+\frac{1}{2\tau}\sum_{k=1}^n d^2_2(\mu_\tau^n,\mu_\tau^{n-1})+\frac{1}{2}\int_0^{n\tau} |\partial E(\overline{\mu}_\tau(t))|^2\;dt\le E(\mu^0_{\tau})=E(\nu_i). 
$$

Moreover, by Proposition 4.1 in \cite{lisini2012cahn}, we know that $\overline{\mu}_\tau(t)\in H^2$ for every $t>0$. Therefore, because of the strong subdifferentiability we know that by Lemma~\ref{lem:subdiff} and Remark~\ref{rem:relationGepssubdiff}
$$
\mathcal{G}(\overline{\mu}_\tau(t))^2\le |\partial E(\overline{\mu}_\tau(t))|^2=\int_\T |(W'(\overline{\mu}_{\tau}(t))-\overline{\mu}_{\tau xx}(t))_x|^2\overline{\mu}_\tau(t)\;dx \;\;\forall t>0.
$$
We deduce 
$$
E(\mu^n_\tau)+\frac{1}{2\tau}\sum_{k=1}^n d^2_2(\mu_\tau^n,\mu_\tau^{n-1})+\frac{1}{2}\int_0^{n\tau} \mathcal{G}(\overline{\mu}_\tau(t))^2\;dt\le E(\nu_i). 
$$
and the Energy inequality \eqref{eq:epsgradflow} follows by taking the limit $\tau\to0$. More precisely, the metric derivative term in the Energy inequality \eqref{eq:epsgradflow} follows exactly as in the proof of Theorem 2.3.3 in \cite{ambrosio2006gradient}. The term involving $\mathcal{G}$ follows from Fatou's Lemma and from the lower semicontinuity proven in Lemma~\ref{lem:lscG} below, using that $\overline{\mu}_\tau(t)\in H^2(\T)\cap C^3_{loc}(\{\overline{\mu}_\tau(t)>0\})$ for every $t>0$ and that $\overline{\mu}_\tau(t)\to \nu(t)$ for every $t\ge0$.
\end{proof}

\begin{lemma}\label{lem:lscG}
Given $\{\mu_n\}_{n\in\N}$, such that $\mu_n\in H^2\cap C^3_{loc}\{\mu_n>0\}$, $\sup_{n\in\N}|\mu_n|_{H^1}<C$ and that $\mu_n\to\mu$ in $\W$, then
$$
\liminf_{n\to\infty}\mathcal{G}^\eps(\mu_n)\ge\mathcal{G}^\eps(\mu)
$$
\end{lemma}
\begin{proof}[Proof of Lemma~\ref{lem:lscG}]
Without loss of generality, we will assume that the potential $W=0$ and $\eps=1$ and that $\liminf_{n\to\infty}\mathcal{G}(\mu_n)<\infty$. We can always take $\mu_{n_i}$, such that 
$$
\lim_{i\to\infty}\mathcal{G}(\mu_{n_i})=\liminf_{n\to\infty}\mathcal{G}(\mu_n)\mbox{ and }\sup_i\mathcal{G}(\mu_{n_i})\le C\mbox{for some $C$}.
$$
From now on, we drop the dependence on $i$. 

Because the $\mu_n$ are probability measures, which are uniformly bounded in $H^1$, we know that
$$
\sup_n||\mu_n||_\infty<C.
$$

Let $g_n\in\mathcal{T}(\mu_n)$ be such that $||g_n||_2=\mathcal{G}(\mu_n)$ (see \eqref{eq:calGeps}) (we can always find such a $g_n$, because $\mathcal{T}(\mu_n)$ is closed) and by definition,
$$
\int |G(\mu_n)_x|^2\le ||\mu_n||_\infty ||g_n||_2^2=||\mu_n||_\infty \mathcal{G}(\mu_n)^2\le C.
$$
Moreover, as
$$
\int G(\mu_n)=\frac{3}{2}|\mu_n|_{H^1}^2<C,
$$
we conclude that
$$
\sup_n ||G(\mu_n)||_{H^1(\T)}<C,
$$
in particular $G(\mu_n)$ is bounded in $L^\infty$ and in $C^\frac{1}{2}(\T)$ uniformly in $n$.

Now, as $\mu_n\in H^2(\T)\subset C^{1,\frac{1}{2}}(\T)$, we know that if $\mu_{nx}(x_0)\ne0$, then $\mu_n(x_0)>0$. So, if $x_0$ is a max of $|\mu_{nx}|$, then, by the hypothesis that $\mu_n\in C^3_{loc}\{\mu_n>0\}$, we have enough regularity to assure that $\mu_{nxx}(x_0)=0$. Then, we can bound
$$
||\mu_{nx}||_\infty^2=|\mu_{nx}(x_0)|^2=2G^\lambda(x_0)\le2||G^\lambda||_\infty\le C.
$$
Therefore,
$$
\sup_n||\mu_{nx}||_\infty\le C,
$$
so we can conclude that $\mu$ is a Lipschitz function and
$$
||\mu||_{Lip}\le C.
$$

Up to subsequence, we know that
$$
G(\mu_n)\to H\mbox{ in $C^\alpha$ for all $\alpha<\frac{1}{2}$,}
$$
and
$$
g_n\to g\mbox{ weakly in $L^2$.}
$$
Because
$$
G(\mu_n)_x=\sqrt{\mu_n}g_n
$$
and $\mu_n\to\mu$ uniformly, we can pass to the limit in the sense of distributions to get
$$
H_x=\sqrt{\mu}g.
$$
Moreover, we know that
$$
||g||_2\le\liminf||g_n||_2,
$$
so it only remains to prove that
$$
g\in\mathcal{T}(\mu),
$$
or, equivalently, that
$$
H=G(\mu).
$$
The rest of the proof is devoted to proving this equality.

Because $\mu_n\in C^3_{loc}(\{\mu_n>0\})$, we can use Remark~\ref{rem:cumbersome} to obtain
$$
\int_{\mu_n>0} \mu_n|\mu_{nxxx}|^2\;dx\le\mathcal{G}(\mu_n),
$$
then we have, up to subsequence,
$$
\mu_n\to\mu \in C^{2}(\{\mu>\lambda\}).
$$
Therefore,
$$
G(\mu)=H\mbox{   in }\{\mu>0\}.
$$
Of course, the set $\{\mu=0\}$ requires a more delicate argument.

First, we prove that $G(\mu)=0$ a.e. in $\{\mu=0\}$. By rewriting
\begin{equation}\label{eq:G}
G(\mu_n)=-\left(\frac{\mu_n^2}{2}\right)_{xx}+\frac{3}{2}|\mu_{nx}|^2,
\end{equation}
and using the fact that $\mu_n$ is uniformly Lipschitz, then we can say that
$$
\sup_n ||\mu_n^2||_{W^{2,\infty}}<C,
$$
therefore
$$
\mu^2\in W^{2,\infty}.
$$
Stampacchia's lemma states that if $f\in W^{1,p}$, then $f_x=0$ a.e. in $\{f=0\}$ (See Lemma A.4. Chapter II \cite{kinderlehrer1980introduction}), hence $G(\mu)=0$ a.e. in $\{\mu=0\}$.

Therefore, we only need to show that $H=0$ a.e. in $\{\mu=0\}$. Instead, we prove something seemingly stronger, more specifically, we prove that if $x_0$ is such that $|H(x_0)|=\delta\ne0$ and $\mu(x_0)=0$, then there exists a non-trivial interval $(a_0,b_0)$ around $x_0$ such that $H=G(\mu)$ in $(a_0,b_0)$. The rest of the proof is devoted to proving this last statement.

Let $x_0$ be such that $|H(x_0)|=\delta\ne0$ and $\mu_0(x_0)=0$, then since $G(\mu_n)$ converges to $H$ uniformly there exists $n_0$, such that $n>n_0$ implies
$$
|G(\mu_n)(x_0)|\ge\frac{\delta}{2}.
$$
Given $\beta>0$, to be chosen later, we consider the open sets
$$
A^n_\beta=\{x:\mu_n<\beta\},
$$
and
$$
A^\infty_\beta=\{x:\mu<\beta\}=\cup_i (a_i^\beta,b_i^\beta),
$$
written as the union of its connected components. From now on, we suppress the dependence on $\beta$ on the end points of the intervals.

Since $x_0\in A^\infty_\beta$, there exists a unique $i_0$ which we take to be $0$, such that
$$
x_0\in (a_0,b_0),
$$
and 
$$
\mu(a_0)=\mu(b_0)=\beta.
$$
As $\mu_n\to \mu$ uniformly, then for all $n$ big enough
$$
(a_0,b_0)\subset A^n_{2\beta}.
$$
and
$$
\mu_n(a_0)>\frac{\beta}{2},\;\;\;\ \mu_n(b_0)>\frac{\beta}{2}.
$$
Using the definition of $\mathcal{G}(\mu_n)$, we can bound the oscillations of $G(\mu_n)$ in the set
$$
osc_{(a_0,b_0)}G(\mu_n)\le \int_{a_0}^{b_0}|G(\mu_n)_x|\le \mathcal{G}(\mu_n)^\frac{1}{2} \left(\int_{a_0}^{b_0} \mu_n\right)^{\frac{1}{2}}\le C\sqrt{2\beta}.
$$
Then,
$$
|G(\mu_n)|\ge G(\mu_n)(x_0)-C\sqrt{2\beta}\;\;\;\;\mbox{in}\;\;(a_0,b_0).
$$
By taking $\beta$ small enough, we deduce that
$$
|G(\mu_n)|\ge\kappa>0\;\;\;\;\mbox{in}\;\;(a_0,b_0).
$$

If $\mu_n$ would vanish at any point in $(a_0,b_0)$, it would contradict the hypothesis that $\mu_n\in H^2$. We prove this by contradiction, if assume that $\mu_n$ vanishes at $y_0\in(a_0,b_0)$, then, because $\mu_n\in C^{1,\frac{1}{2}}$, $\mu_{nx}(y_0)=0$ and there exists $\eps_0$ such that
$$
|\mu_{nx}(x)|^2<\kappa\;\;\;\;\mbox{for every $x\in (y_0-\eps_0,y_0+\eps_0).$}
$$
Therefore,
$$
|\mu_n(x)\mu_{nxx}(x)|\ge |G(\mu_n)|-\frac{|\mu_{nx}|^2}{2}\ge \frac{\kappa}{2}\;\;\;\;\mbox{for every $x\in (y_0-\eps_0,y_0+\eps_0)$.}
$$
Finally, using that $\mu_n$ is Lipschitz and $\mu_n(y_0)=0$ we know that
$$
\mu_n(x)\le C(x-y_0).
$$
We deduce that
$$
|\mu_{nxx}(x)|^2\ge \frac{C}{(x-y_0)^2}\;\;\;\;\mbox{for every $x\in (y_0-\eps_0,y_0+\eps_0)$,}
$$
which is not integrable at $y_0$ and thus contradicts the fact that $\mu_n$ is in $H^2$. So, we can conclude that
$$
\mu_n> 0\;\;\;\;\mbox{in}\;\;(a_0,b_0).
$$

Now we apply Theorem 4.3 of Bernis' integral inequalities \cite{bernis1996integral}, which proves that
$$
|v|_{W^{2,3}(a_0,b_0)}\le \int_{a_0}^{b_0}v|v_{xxx}|^2\;dx
$$
for $v\in C^3$, with $v>0$ in $(a_0,b_0)$, such that $v'(a_0)=v'(b_0)=0$. 

Note that we cannot use this result directly because we do not know that $\mu_n'(a_0)=\mu_n'(b_0)=0$. So, to be able to apply the Theorem, we consider $\phi_n$ smooth, such that 
\begin{itemize}
\item $\phi_n(a_0)=\phi_n(b_0)=\frac{\beta}{4}$.

\item $\phi_n'(a_0)=\mu_n(a_0)$,  $\phi_n'(b_0)=\mu_n(b_0)$.

\item $\phi_n< \max\left(0, \frac{\beta}{2}-|\mu_n|_{lip}(x-a_0),\frac{\beta}{2}-|\mu_n|_{lip}(b_0-x)\right)<\mu_n$.
\end{itemize}
Because $||\mu_{nx}||_\infty$ is uniformly bounded, we get that $\phi_n$ is uniformly bounded in $W^{2,3}$ and $H^3$. Moreover, $v_n=\mu_n-\phi_n$ satisfies the hypothesis of \cite{bernis1996integral}, then
$$
|v_n|_{W^{2,3}(a_0,b_0)}
\le
\int_{a_0}^{b_0}(\mu_n-\phi_n)|(\mu_n-\phi_n)_{xxx}|^2\;dx
\le
\int_{a_0}^{b_0}\mu_n|\mu_{nxxx}|^2\;dx+|\mu_n|_\infty|\phi_n|_{H^3}.
$$
Also,
$$
|\mu_n|_{W^{2,3}(a_0,b_0)}\le|v_n|_{W^{2,3}(a_0,b_0)}+|\phi_n|_{W^{2,3}(a_0,b_0)},
$$
so we finally deduce the following uniform bound for $\mu_n$ in the interval $(a_0,b_0)$:
$$
\sup_n||\mu_n||_{W^{2,3}(a_0,b_0)}<C.
$$

This implies, in particular, that up to subsequence, $\mu_n$ converges to $\mu$ uniformly in $C^{1,\alpha}$ and so
$$
\mu_{nx}\to \mu_x \mbox{ uniformly in $(a_0,b_0)$},
$$
which combined with the fact that 
$$
(\mu_n^2)_{xx}\to (\mu^2)_{xx}\mbox{in $\mathcal{D}'$},
$$
yields, by passing to the limit in \eqref{eq:G}
$$
G(\mu_n)\to G(\mu) \mbox{ in $(a_0,b_0)$.}
$$
So, in particular
$$
H=G(\mu) \mbox{ in $(a_0,b_0)$.}
$$
\end{proof}

\bibliographystyle{plain}
\bibliography{Bibliography}
\end{document}